\documentclass[11pt,reqno]{amsart}
\usepackage{amsmath,amsfonts,amssymb,amsthm}
\usepackage[mathcal]{euscript}
\usepackage{mathrsfs}
\usepackage[tmargin=1.4in,bmargin=1.4in,lmargin=1.2in,rmargin=1.2in]{geometry}
\usepackage{graphicx}
\usepackage{tikz}

\graphicspath{{./Images/}}
\usepackage[font=small,format=plain,labelfont=sf,up,textfont=sf,up]{caption}
\usepackage{boxedminipage}

\usepackage{microtype}
\usepackage{bbm}
\usepackage{color}
\usepackage[textsize=small]{todonotes}
\usepackage{enumitem}

\usepackage{thmtools} 
\declaretheorem[numberwithin=section, name=Theorem]{thm}
\declaretheorem[sibling=thm, name=Lemma]{lem}
\declaretheorem[sibling=thm, name=Remark] {rk}
\declaretheorem[sibling=thm, name=Definition]{defi}
\declaretheorem[sibling=thm, name=Proposition]{prop}

\definecolor{halfgray}{gray}{0.55}
\definecolor{webgreen}{rgb}{0,.5,0}
\definecolor{webbrown}{rgb}{.6,0,0}
\definecolor{Maroon}{cmyk}{0, 0.87, 0.68, 0.32}
\definecolor{royalblue}{RGB}{0,0,139}
\definecolor{lightblue}{RGB}{150,230,255}

\definecolor{Black}{cmyk}{0, 0, 0, 0}
\definecolor{pinkish}{RGB}{255, 192, 203}

\usepackage{hyperref}
\hypersetup{%
    colorlinks=true, linktocpage=true, pdfstartpage=1, pdfstartview=FitV,%
    breaklinks=true, pdfpagemode=UseNone, pageanchor=true, pdfpagemode=UseOutlines,%
    plainpages=false, bookmarksnumbered, bookmarksopen=true, bookmarksopenlevel=1,%
    hypertexnames=true, pdfhighlight=/O,%nesting=true,%frenchlinks,%
    urlcolor=webbrown, linkcolor=royalblue, citecolor=webgreen, 
}

\usepackage{tikz}
\pgfdeclarelayer{nodelayer}
\pgfdeclarelayer{edgelayer}
\pgfsetlayers{edgelayer,nodelayer,main}
\tikzstyle{n}=[circle,fill=black,draw=black,line width=0.2 pt,minimum size=0.1 cm ]
\numberwithin{equation}{section}
\newcommand{\ER}{Erd\H{o}s-R\'enyi  \ }

\renewcommand{\P}{\mathbb{P}}
\newcommand{\E}{\mathbb{E}}
\newcommand{\dto}{\overset{d}{\to}}
\newcommand{\deq}{\overset{d}{=}}

\newcommand{\pto}{\overset{\P}{\to}}
\newcommand{\Bin}{\mathsf{Bin}}
\newcommand{\Ber}{\mathsf{Ber}}
\newcommand{\Tp}{\Theta_{\scriptscriptstyle\P}}
\newcommand{\Var}{\mathsf{Var}}
\newcommand{\op}{o_{\scriptscriptstyle \P}}
\newcommand{\Op}{O_{\scriptscriptstyle \P}}
\newcommand{\Ball}{\partial \mathcal B}

\newcommand{\cluster}{\mathcal{C}}
\newcommand{\Act}{\mathcal{A}}
\newcommand{\Dead}{\mathcal{D}}
\newcommand{\Opp}{\mathcal{Q}}
\newcommand{\Fil}{\mathscr F}
\newcommand{\Ucal}{\mathcal{U}}
\newcommand{\Wcal}{\mathcal{W}}
\newcommand{\Dex}{\Dead ^+}

\newcommand{\indi}{\mathbbm{1}}
\newcommand{\Ical}{\mathcal{I}}
\newcommand{\Ver}{\mathcal V}
\newcommand{\Ed}{\mathcal E}

\begin{document}

%%%%%%%%% TITLE PAGE %%%%%%%%%%%%

\title[Critical scaling limits of the random intersection graph]
{Critical scaling limits of the random intersection graph}

\begin{abstract}
We analyse the scaling limit of the sizes of the largest components of the Random Intersection Graph $G(n,m,p)$ close to the critical point $p=\frac{1}{\sqrt{nm}}$, when the numbers $n$ of individuals and $m$ of communities have different orders of magnitude. We find out that if $m \gg n$, then the scaling limit is identical to the one of the \ER Random Graph (ERRG), while if $n \gg m$ the critical exponent is similar to that of  Inhomogeneous Random Graphs with heavy-tailed degree distributions, yet the rescaled component sizes have the same limit in distribution as in the ERRG. This suggests the existence of a wide universality class of inhomogeneous random graph models such that in the critical window the largest components have sizes of order $n^{\rho}$ for some $\rho \in (1/2,2/3]$, which depends on some parameter of the graph.
\end{abstract}

\author[L.\ Federico]{Lorenzo Federico}
\email{lorenzo.federico@warwick.ac.uk}
\maketitle

\section{Introduction}

\subsection{The model}

In this paper, we study the phase transition of the Random Intersection Graph  (RIG) $G(n,m,p)$, that is, the intersection graph of the random bipartite graph $K_p(n,m)$ obtained by uniform edge percolation on the complete bipartite graph. 
The most common use of Random Intersection Graphs is to model networks that have a community structure, that is, networks in which there exist some sets of vertices (communities) that have many more connections among themselves than with the rest of the graph. This is often the case in social networks, as people tend to form many kinds of communities in real life, for example, people working at the same place, sports clubs, political associations, etc., and usually two people meet each other when they both belong together to some of these social groups. To reflect this structure, the RIG is built from a random bipartite graph, defining the vertices on one side as individuals, and the vertices on the other as communities. When the edge between an individual and a community is present, it means that the individual belongs to that community.
The RIG has as vertex set the set of individuals and every edge between two individuals is present if and only if there exists at least one community that contains both.

The RIG has been very popular as one of the simplest procedures to produce a random graph with clustering, i.e., a model in which vertices that share common neighbours are more likely to be directly connected  (see \cite{BloGodJawKurRyb15, BloJawKur13, DefKet09} for a discussion on the clustering coefficient of different RIGs), without resorting to complicated geometric models, which are typically much harder to analyse.  Clustering is a feature that appears in many real-world networks \cite{GirNew02} and that is completely absent in the classical models, such as the \ER random graph  $G(n,p)$ (ERRG), where all edges are present independently. In recent years there has been great emphasis on the creation of new models that can describe clustered graphs. Many are based on geometric constructions (see e.g. \cite{DeiHofHoo13, KriPapKitVahBog10}) and researchers are investigating whether all graphs with clustering can be expressed in terms of an underlying geometric structure \cite{Kri16}. In other cases the community structure is assigned explicitly (see e.g. \cite{DhaLeuMuk17, HofLeuSte16}). Still, the RIG in its different versions remains a crucial tool to model clustering and community structure in graphs (see \cite{HofKomVad18,KarLeeLes18}).

We now give a formal definition of the Intersection Graph of a bipartite graph $G$:

\begin{defi}[Intersection graph]\label{defi:intersection}
Given a bipartite graph $G$ over the sets of vertices $\Ucal, \Wcal$, we define the vertex and edge sets of the intersection graph $G'$ associated to $G$ as follows:
	\begin{equation}\begin{split}
	\Ver(G')&:=\Ucal,\\
	\Ed(G')&:=\{ \{v_1,v_2\}: \exists w\in  \Wcal \ s.t. \ \{v_1,w\},\{v_2,w\} \in \Ed(G)\}.
	\end{split}\end{equation}

\end{defi}
Note that every bipartite graph has a corresponding intersection graph and every graph can be built as the intersection graph of some bipartite graph, which is typically not unique \cite{ErdGooPos66}.
We define the RIG $G(n,m,p)$, as in \cite{KarSchSin99}, as the intersection graph over the graph $K_p(n,m)$:

\begin{defi}[Construction of the random bipartite $K_p(n,m)$]

We define the random bipartite graph $K_p(n,m)$, for any $n,m \in \mathbb{N}$, $p \in [0,1]$, as the graph obtained as follows:
\begin{itemize}
\item Build the complete bipartite graph $K(n,m)$ over the sets $\Ucal=[n]$, $\Wcal=[m]$.
\item Remove each edge with probability $1-p$ independent of each other.
\end{itemize}
\end{defi}
Note that in the construction of $G(n,m,p)$ all the randomness is related to the realisation of $K_p(n,m)$, and the construction of the intersection graph from a given bipartite graph is a deterministic procedure.

In the literature there exist many more versions of the RIG, built starting from different bipartite random graphs, such as the uniform random intersection graph $G(n,m,d)$, in which all communities have size $d$ and the RIG built from the bipartite configuration model (see \cite{Ryb11} and \cite{HofKomVad18}, respectively, for more details). Moreover, also RIGs in which the communities are not represented by cliques but other, random or deterministic, graphs have been studied, for example in \cite{HofKomVad18, HofKomVad19} and \cite{KarLeeLes18}. We will stick to the study of $G(n,m,p)$, but we have reasons to believe that our results and proof techniques might be valid in many different settings as long as the  sizes of the communities are not too inhomogeneous and they are all internally connected.

\subsection{The phase transition in the Random Intersection Graph}
 Behrisch proved in \cite{beh07} that $G(n,m,p)$ undergoes a phase transition at the critical point $p_c:=\frac{1}{\sqrt{nm}}$, in the limit as $n,m \to \infty$.  The phase transition for the size of the largest component $\cluster_1$ has a different shape depending on the relative scaling of $n$ and $m$:

\begin{thm}[Phase transition in the RIG \cite{beh07}] Fix $\alpha >0$. Consider the graph $G(n,m,p)$. Then in the limit as $n \to \infty$, $m=n^\alpha$, and $p=\frac{\mu}{\sqrt{nm}}$,

\begin{itemize}
	\item  for $\alpha \geq 1$,

	\begin{equation}
	|\cluster_1| =\left\{
	\begin{aligned} 
	&\Tp (\log n)&  \text{ if } \mu<1, \\ 
	&\Tp (n)  &\text{ if } \mu >1.\end{aligned}
\right.
	\end{equation}
	\item  for $\alpha < 1$,

	\begin{equation}\label{eq:superca}
	|\cluster_1| =\left\{
	\begin{aligned} 
	&\Tp (\log m\sqrt{n/m})&  \text{ if } \mu<1, \\ 
	&\Tp (\sqrt{ nm})  &\text{ if } \mu >1.\end{aligned}
\right.
	\end{equation}
\end{itemize}
\end{thm}\noindent
 Fill et al. \cite{FilSchSin00} identified the connectivity threshold and the threshold above which $G(n,m,p)$ is a complete graph with high probability (w.h.p.).
\medskip
 
 We study  the scaling limit of the sequence of component sizes, $(|\cluster_i|)_{i \geq 1}$ arranged in decreasing order, when $\alpha<1$, we prove it together with the limit of the sequence of edge-size of the largest components $(|\Ed(\cluster_i)|)_{i \geq 1}$ . We prove these limits in terms of the $\ell^2$-topology and $\ell^2_\searrow$-topology. We define the $\ell^2_\searrow$-topology as the topology induced over the set
\begin{equation}
\ell^2_\searrow=\Big\{ (x_i)_{i\geq 1}: x_i\geq x_{j} \text{ if } i\geq j; \sum_{i=1}^{\infty}x_i^2<\infty\Big\},
\end{equation}
by the $\ell^2$-distance.

We follow a similar approach and aim for similar results as in the work of Aldous \cite{Ald97} about the ERRG. Aldous, defining $\sigma_i$ as the number of surplus edges of the $i$-th largest component, i.e., the number of edges that have to be removed from it to make it a tree, proved the following theorem:

\begin{thm}[Scaling limit for critical ERRG \cite{Ald97}] For the graph $G(n,p)$, when $n \to \infty$, for every $\lambda \in \mathbb R$, $p=n^{-1}(1+\lambda n^{-1/3})$, the following limit holds, for suitable limiting sequences of non-degenerate random variables $\big(\mathbf{C}_i^{\lambda},\boldsymbol{\sigma}_i^{\lambda}\big)_{i\geq 1}$,
	\begin{equation}\label{eq:alfageq}
	\big( n^{-2/3}|\cluster_i|,\sigma_i\big)_{i\geq 1} \dto \big(\mathbf{C}_i^{\lambda},\boldsymbol{\sigma}_i^{\lambda}\big)_{i\geq 1},
	\end{equation}
in the $\ell^2_\searrow\times \mathbb N^\infty$-topology.
\end{thm}

In this paper we analyse the critical behaviour of the RIG, summarized in the following theorem, which is our main result:

\begin{thm}[Scaling limit for critical RIG]\label{thm:main}For the graph $G(n,m,p)$, the following limits hold when $n \to \infty$, with $m=n^\alpha$, :

\begin{enumerate}[label=(\roman*)]
\item  If $\alpha > 1$, $\lambda \in \mathbb R$, $p=p_c(1+\lambda n^{-1/3})$, then, 
	\begin{equation}\label{eq:alfageq}
	\big( n^{-2/3}|\cluster_i|\big)_{i\geq 1} \dto \big(\mathbf{C}_i^{\lambda}\big)_{i\geq 1},
	\end{equation}
in the $\ell^2_\searrow$-topology.
\item  If $\alpha \in (0,1)$, $\lambda \in \mathbb R$, $p=p_c(1+\lambda m^{-1/3})$, then,
	\begin{equation}\label{eq:alfamin}
	\big( n^{-1/2-\alpha/6}|\cluster_i|, n^{-1+\alpha/3}|\Ed(\cluster_i)|\big)_{i\geq 1} \dto \big(\mathbf{C}_i^{\lambda},\mathbf{C}_i^{\lambda}/2\big)_{i\geq 1},
	\end{equation}
in the $\ell^2_\searrow \times \ell^2$-topology.
\end{enumerate}
The limiting variables $ \big(\mathbf{C}_i^{\lambda}\big)_{i\geq 1}$ are the same as the ones for the component sizes of a critical ERRG, with $p=\frac{1+2\lambda n^{-1/3}}{n}$ given in \cite{Ald97}.
\end{thm}

Similarly, we also obtain the critical scaling limit of the random bipartite graph $K_p(n,m)$, which is interesting in its own right. In this case, each component contains vertices belonging to both sides of $K_p(n,m)$, not only to the side labelled as individuals, but the leading contribution always comes from the number vertices on the larger side.
 
 \begin{thm}
 [Scaling limit for critical random bipartite graph]\label{thm:main2}Consider $K_p(n,m)$, with $p=p_c(1+\lambda n^{-1/3})$. Define for every $i \geq 1$, $\cluster_i$ as the $i$th largest component, then in the limit as $n \to \infty$, $m=n^\alpha$, $\alpha > 1$, 

	\begin{equation}\label{eq:bipa}
	\big( n^{-1/6-\alpha/2}|\cluster_i|\big)_{i\geq 1} \dto \big(\mathbf{C}_i^{\lambda}\big)_{i\geq 1},
	\end{equation}
	in the $\ell^2_\searrow$-topology, where the limiting variables $ \big(\mathbf{C}_i^{\lambda}\big)_{i\geq 1}$ are the same as the ones for the critical ERRG with $p=\frac{1+2\lambda n^{-1/3}}{n}$ given in \cite{Ald97}.

 \end{thm}
We discuss the case $\alpha <1$ only for the RIG and not for the random bipartite graph, because in the graph $G(n,m,p)$ the sets of individuals and communities play very different roles, while in $K_p(n,m)$ the roles of $n$ and $m$ are exchangeable, since $K(n,m)$ is isomorphic to $K(m,n)$ and thus we can always without loss of generality assume that $n \leq m$.  It is worth noting that the critical exponent for component sizes in the RIG  may be different from the one for the ERRG, but the limiting variables are the same, with just a rescaling by a factor $2$ in the parametrization of the critical window. This factor $2$ comes from the binomial expansion 
\begin{equation}
p^2=p_c^2(1+\lambda n^{-1/3})^2=p_c^2(1+2\lambda n^{-1/3}+\lambda^2 n^{-2/3}),
\end{equation}
since the critical point, is determined by the average number of vertices that are at distance $2$ from a randomly chosen vertex in $K_p(n,m)$. We show this in more detail in Chapter \ref{sec:explrel}, when we design an exploration algorithm that inspects at every step the $2$-neighbourhood of a given vertex and use it as a key tool to prove our main theorems.

It is also relevant to note that in Theorem \ref{thm:main} (2) we talk about total number of edges instead of surplus edges (as it is common in most of the literature about scaling limits of critical random graphs, see, e.g \cite{Ald97, DhaHofLeeSen16a, BhaHofLee10a}) as in this case the number of edges is of higher order of magnitude compared to the number of vertices in a typical large critical component, and consequently, the number of  surplus  edges is very close to the total number of edges. Also, here we obtain $\ell^2$-convergence for the (rescaled) number of surplus edges, which does not hold for any of the already studied cases.

 The results about component sizes in Theorems \ref{thm:main} and \ref{thm:main2} are strongly linked, since the connected components of $G(n,m,p)$ are just the restrictions of the components of $K_p(n,m)$ to the set $\Ucal$.

% Note that  in the case $\alpha=1$ $K_p(n,m)$  falls into the class analysed in \cite{BhaBroSenWan14} and thus the scaling limit is already known. In such case 
%
%	\begin{equation}
%	\big( n^{-2/3}|\cluster_i^\lambda|\big)_{i\geq 1} \dto \big(\mathbf{C}_i^{\lambda}\big),
%	\end{equation}
%\todoin{Reference might be wrong}
%	but the limit variables are different from the ones of the ERRG.
 Moreover, Fill et al. \cite{FilSchSin00} have proved that if $\alpha > 6 $, then $G(n,m,p)$ converges in total variation distance to an ERRG.
\noindent
	\subsection*{ Notation.}
All limits in this paper are taken as $n\to \infty$ unless stated otherwise. Since most of them are joint limits in $n$ and $m$, $m$ will be specified as a function of $n$ when the precise relation is relevant. 
For asymptotic statements we use the following notation:
\begin{itemize}
\item Given a sequence of events $(\mathcal{A}_n)_{n \geq 1}$ we say that $\Act_n$ happens \emph{with high probability (w.h.p.)} if $\P(\mathcal{A}_n) \to 1$.
\item Given the random variables $(X_n)_{n \geq 1}, X$, we write $X_n \dto X$ and $X_n \pto X$ to denote convergence in distribution and in probability, respectively.
\item For sequences of (possibly degenerate) random variables  $(X_n)_{n \geq 1}$, $(Y_n)_{n \geq 1}$, we write $X_n=O(Y_n)$ if the sequence $(X_n/Y_n)_{n \geq 1}$ is bounded almost surely; $X_n=o(Y_n)$ if $X_n/Y_n \to 0$ almost surely; $X_n =\Theta(Y_n)$ if $X_n=O(Y_n)$ and $Y_n=O(X_n)$ .
\item Similarly, for sequences $(X_n)_{n \geq 1}$, $(Y_n)_{n \geq 1}$ of (possibly degenerate) random variables, we write $X_n=\Op(Y_n)$ if the sequence $(X_n/Y_n)_{n \geq 1}$ is tight; $X_n=\op(Y_n)$ if $X_n/ Y_n \pto 0$; and $X_n =\Tp(Y_n)$ if $X_n=\Op(Y_n)$ and $Y_n=\Op(X_n)$.  
\end{itemize}

We write $\Ber(p)$ for a Bernoulli random variable with success probability $p$ and $\Bin(n,p)$ for a binomial random variable with $n$ independent trials and success probability $p$.
Given two random variables $A$ and $B$ we write that $A\succeq B$ if, for all $x \in \mathbb R$, $\P (A\geq x) \geq \P(B\geq x)$, and $A \overset{d}{=}B$ if, for all $x \in \mathbb R$, $\P (A\geq x) = \P(B\geq x)$.
We use calligraphic letters to denote sets and capital letters to denote their cardinalities (for example $D$ is the cardinality of $\Dead$).

For any, deterministic or random, graph $G$ we write $\Ed(G)$ to denote its edge set, $\Ver (G)$ to denote its vertex set and $\cluster_G(v)$ for the connected component of the vertex $v$ in $G$. Often we will use the convention that $|G|=|\Ver(G)|$.

We will use the standard abbreviation i.i.d. for independent identically distributed and a.s. for almost surely.

Given a real-valued random variable $X$ and a $\sigma$-algebra $\Fil$, we define $(X\mid \Fil)$ as the conditional distribution of $X$ with respect to $\Fil$. Note that $(X\mid \Fil)$ is a random variable which takes values in the space of real-valued random variables and that it is measurable with respect to $\Fil$.

\subsection{Comparison with other models}

The main result that emerges from this paper is that the size of the largest clusters of the RIG at criticality is in the universality class of the critical ERRG \cite{Ald97} if $\alpha > 1$, while it is in a different one if $\alpha <1$. We can compare the behaviour of the RIG with many other inhomogeneous random graph models, whose critical phase has been studied in detail, starting with the work of Aldous and Limic \cite{AldLim98} We notice a strong analogy with the behaviour of the rank-1 inhomogeneous random graphs (IRG) \cite{BhaBroSenWan14,BhaHofLee10a,BhaHofLee12,BroDuqWan18} and the configuration model (CM) \cite{DhaHofLeeSen16b,DhaHofLeeSen16a} with power-law degree distributions. These models show a critical behaviour that depends on the precise choice of the exponent $\tau$ of the power-law in a similar way as the critical behaviour of the RIG depends on the exponent $\alpha$.

We define the critical exponent $\rho $ as the limit in probability of $\log|\cluster_1|/\log n$.
Both for the IRG and the RIG there exists a region for the parameters ($\tau >4$ and $\alpha >1$ respectively) in which the phase transition has the same critical exponent $\rho = 2/3$ as the ERRG, which is considered the mean-field model for random graphs, as it is the one with the highest symmetry and homogeneity. Instead for $\tau \in (3,4)$ and $\alpha \in (0,1)$ we have a non-mean-field critical exponent $\rho\in (1/2,2/3)$, and, finally, we see that for $\tau<3$ and $\alpha =0$, respectively, there is no phase transition.

It is important to notice that the structure of the IRG and the RIG are very different, since in the IRG a few \emph{hubs} (the vertices with the largest degrees, of order $n^{\frac{1}{\tau -1}}$) are the ones that determine the structure of the random graph, with the other vertices playing a minor role, while in the RIG all the communities have size $\Tp (n^{\frac{1-\alpha}{2}})$ and the source of inhomogeneity is the existence of two different kinds of vertices in $K_p(n,m)$. This difference is seen in the variables to which critical component sizes converge after rescaling. The results in \cite{BroDuqWan18} by Broutin et al. about the IRG do not require the assumption that the degree distribution converges to a power law. Their result imply that it would be possible to generate versions of the IRG (and maybe the CM) with non-mean-field critical exponent but the same limit variables as in the ERRG and the RIG. This could be achieved imposing the existence of very few vertices of very high degree and a vast majority of vertices of very low degree, with no middle ground, as it happens in $K_p(n,m)$. These observations suggest the existence of a bigger class of inhomogeneous random graph models such that $\rho= 2/3$ when the inhomogeneity is not too strong, and $\rho\in (1/2,2/3)$ when the graph is highly inhomogeneous, and it would be interesting to analyse even more inhomogeneous models to investigate whether this phenomenon is indeed universal.
\bigskip
\begin{center}
\begin{tabular}{| l | c | c | }
\hline   & IRG & RIG  \\
 \hline $\rho =2/3$ & $\tau > 4$ & $\alpha > 1$ \\
  $1/2 < \rho < 2/3$ & $3<\tau<4$ & $0<\alpha<1$ \\
  No phase transition & $\tau <3$ & $\alpha =0$\\
  \hline
\end{tabular}
\end{center}

Another analogy we see with IRG and CM is that as the inhomogeneity of the graph increases (i.e., as $\tau$ and $\alpha$ get smaller) the largest subcritical components get bigger, compare for example the results of Janson about the CM \cite{Jan08} and of Behrisch about the RIG \cite{beh07}. This has a clear interpretation as a consequence of the increase in the maximal degree, which is $\Tp(n^{\frac{1}{\tau -1}})$ in the IRG and CM and $\Tp (n^{\frac{1-\alpha}{2}})$ in the RIG, and corresponds, up to respectively a constant and a logarithmic term, to the size of the largest subcritical component. 
On the other hand, in the supercritical phase we see that in the IRG and CM the largest supercritical component has size $\Tp (n)$ for all $\tau$, while in the RIG it depends on the choice of $\alpha$, as we have seen in \eqref{eq:superca}. 

For the IRG and the CM also  the scaling limit of the large critical components seen as metric spaces have been investigated (see \cite{AddBroGol12,BhaSenWan17,BhaDhaHofSen19}), and in particular, it is known that the typical distances between vertices inside a large critical connected component scale as $n^{\eta}$, with $\eta = (\tau-3)/(\tau -1)$, which means that they get smaller as the inhomogeneity increases. It is reasonable to expect a similar behaviour in the RIG as $\alpha \to 0$, since when $\alpha =0$ (i.e. when $m=1$ for all $n$), the largest component  is a clique, which has diameter $1$.

Note that if we choose $\alpha=1$, i.e. $m=n$, $K_p(n,n)$  is a special case of the stochastic block model, whose critical scaling limit has been studied by Bhamidi et al. \cite{BhaBroSenWan14} in quite broad generality. Unfortunately, due to the requirement that edges between vertices on the same side of $K_p(n,n)$ are not present deterministically, the results from \cite{BhaBroSenWan14} do not apply directly to $K_p(n,n)$ as, for technical reasons, Bhamidi et al. required that all the edge probabilities were non-zero. We think that the problem of the stochastic block model with some deterministically vacant edges can be solved in much broader generality than just the special case $K_p(n.n)$ with techniques from \cite{BhaBroSenWan14}.

	\subsection{Outline of the proof}
The proof of the main theorems of this paper is carried out using the graph exploration argument invented originally by Aldous in \cite{Ald97} to study the critical ERRG. 

In Section \ref{sec:explrel} we design a two-step exploration process on the graph $K_p(n,m)$ which is taylored to the analysis of the corresponding RIG. In particular, in Section \ref{sec:expproc} we describe the exploration process and all the sets related to it, in Section \ref{sec:adaproc} we define the stochastic process $(S_n^\lambda(k))_{k \geq 1}$ that keeps track of the number of active vertices in the exploration, and explain how we use the symmetry of $K(n,m)$ to run the exploration starting always from the smaller side, even when it is the community side. This approach might seem counterintuitive, as the vertex set of $G(n,m,p)$ contains only the individuals, but makes the analysis of the process much smoother. In Section \ref{sec:brownian}, we prove the main theorems assuming the convergence of a rescaled version of this process to a Brownian motion with parabolic drift, which is proved later in Section \ref{sec:mclt}. 

At the beginning of Section \ref{sec:mclt} we state the Martingale Functional Central Limit Theorem (MFCLT), closely adapting its formulation given in \cite{FedHofHolHul16a}, and then we prove that $(S_n^\lambda(k))_{k \geq 1}$ satisfies the conditions required to apply this MFCLT (Sections \ref{sec:continuity}, \ref{sec:parabolic} and \ref{sec:variance}) and thus $(S_n^\lambda(k))_{k \geq 1}$ converges to a Brownian motion after appropriate rescaling and recentering. 

In Section \ref{sec:edge} we study the process that counts the number of edges explored, which is necessary to prove the joint scaling limit in \eqref{eq:alfamin}.

\section{Exploration of the graph and related objects}
\label{sec:explrel}
In this paper, as in many other works whose goal is to prove a scaling limit for critical random graphs, we will follow the approach invented by Aldous in \cite{Ald97} to study the critical ERRG. We define an appropriate \emph{exploration algorithm} on $K_p(n,m)$ in which the vertices of a component are explored sequentially, starting from a random vertex, and then we analyse the process that counts the number of active vertices (i.e., vertices that have been found by the process but whose neighbourhoods have not been inspected yet). We prove that an appropriately rescaled version of such process converges in distribution to a Brownian motion. In the rest of this section, we will make this precise.

\subsection{The exploration process}
\label{sec:expproc}

In this section we define a process that explores the graph $K_p(n,m)$ keeping track of the information that is relevant to compute the number of vertices and edges in each connected component. This is a standard tool in most proofs of critical scaling limits of random graphs, and we will follow a similar approach to the one used in \cite{FedHofHolHul16a}. Since the important object of interest is the RIG associated to $K_p(n,m)$, this will be a two-step exploration, in the sense that at each step we look for the vertices that are at distance $2$ in $K_p(n,m)$ from the vertex $v$ that we are exploring from, which correspond to the direct neighbours of $v$ in the RIG $G(n,m,p)$. Given two vertices $v,w$ in a graph $G$, we define $d_{G}(v,w)$ as the graph distance between $v$ and $w$, with the convention that $d_G(v,w)=\infty$ if they are in different connected components. For any $r \in \mathbb{N}$ we define the sphere centered in $v$ of radius $r$ as

	\begin{equation}
	\Ball_r(v) := \{w \in \Ucal \cup \Wcal: d_{K_p(n,m)}(v,w)=r\}.
	\end{equation}
Note that if $r$ is odd, then all the vertices in $\Ball_r(v)$ belong to the opposite side of $K_p(n,m)$ with respect to $v$, while if $r$ is even, then they are on the same side. Moreover, given a set $\mathcal{H}\subset (\Ucal \cup \Wcal )$, we write

	\begin{equation}\label{eq:ballout}
	\Ball_r(v,\mathcal H) := \{w\in (\Ucal \cup \Wcal)\setminus \mathcal H: d_{K_p(n,m)\setminus \mathcal{H}}(v,w)=r\},
	\end{equation}
	where $K_p(n,m)\setminus \mathcal{H}$ is the subgraph of $K_p(n,m)$ induced by $ (\Ucal \cup \Wcal )\setminus \mathcal{H}$.
We can now properly define the exploration process on $K_p(n,m)$:

\begin{defi}[Two-step exploration process]\label{def:exploration}Consider the graph $K_p(n,m)$ and pick a vertex $v_0\in \Ucal \cup \Wcal$ according to any arbitrary rule. We define the two-step exploration process starting from the vertex $v_0$ as the process $\big( \Act (k),\Dead (k), \Opp (k)\big)_{k \geq 0}$, with update rules as follows:

\begin{itemize}
\item[Initialize:] Define the \emph{active}, \emph{dead} and \emph{opposite} vertex sets at time $k=0$ as

	\begin{equation}
	\Act (0) :=\{v_0\}, \quad \Dead (0):=\varnothing , \quad \Opp (0) := \varnothing.
	\end{equation}
\item[Step $k \geq 1$:] If $\Act (k-1) \neq \varnothing$, then choose a vertex $v_k \in \Act (k-1)$ uniformly at random, else, pick the vertex $v_k$ uniformly at random in $\Ucal\setminus \Dead (k-1)$ if $v_0 \in \Ucal$ or in $\Wcal\setminus \Dead (k-1)$ if $v_0 \in \Wcal$ and update as follows:

	\begin{equation}\begin{split}
	\Act (k) &:=(\Act (k-1)\setminus \{v_k\}) \cup \Ball_2(v_k,\Opp(k-1)\cup \Dead (k-1)),\\
	\Dead (k) &:= \Dead (k-1) \cup \{v_k\},\\
	\Opp (k)& := \Opp (k-1) \cup \Ball_1(v_k).
	\end{split}\end{equation} 

\item[Terminate:] If $\Dead (k)=\Ucal$ or $\Dead (k) = \Wcal$, then terminate the process.
\end{itemize}
\end{defi}

Note that, by construction, all the vertices in $\Act (k)$ and $\Dead (k)$ are on the same side as $v_0$, while those in $\Opp (k)$ are on the opposite side.
The exploration is constructed in such a way that every time $k$ for which $\Act (k)=\varnothing$ we have completed the exploration of a connected component of $K_p(n,m)$.
 For brevity we define the set of all vertices found by time $k$ by the exploration process $\Dex (k) := \Dead (k) \cup \Act (k)$, since it will be used frequently in the rest of the paper.

Moreover we define the edge set related to the exploration as follows:

\begin{defi}[Edge-set process related to the two-step exploration process] \label{def:edge}Consider the two-step exploration process of the graph $K_p(n,m)$ from Definition \ref{def:exploration}. We define the edge-set process $(\Ed(k))_{k\geq 0}$ such that for every $k \geq 0$, $\Ed(k)$ the set of edges of the RIG over the subgraph of $K_p(n,m)$ induced by $\Dead(k)\cup \Opp(k)$.
\end{defi}

This edge-set process will be useful in Section \ref{sec:edge} to count the edges in the RIG as the exploration process goes on, and thus to derive the scaling limit for the number of edges in large critical components.

\subsection{The adapted process}
\label{sec:adaproc}
  In order to derive relevant informations about the size of the connected components from our exploration, we define the stochastic process $\big(S^\lambda_n(k)\big)_{k\geq 0}$, adapted to the filtration $\big( \Fil_n (k)\big)_{k\geq 0}$ generated by the exploration itself, as follows:

	\begin{equation}\label{eq:adapted}\begin{split}
	S_n^\lambda (0):=1, &\qquad S_n^\lambda (k):=S_n^\lambda (k-1)+X_n(k)\\ X_n(k)&:=\big|\Ball_2(v_k,\Opp(k-1)\cup \Dex (k-1))\big|-1.
	\end{split}\end{equation}
 Moreover we define the process $\big(R^\lambda_n(k)\big) _{k\geq 0}$ as the reflected version of $\big(S^\lambda_n(k)\big)_{k\geq 0}$, i.e.,

	\begin{equation}
	 R^\lambda_n(k):= S^\lambda_n(k)-\min_{j< k}  S^\lambda_n(k)+1.
	\end{equation}

The majority of the rest of this paper will be devoted to the study of these processes, from which we will be able to recover important information about the structure of the critical RIG.
This is possible since, from the definitions we gave, $R^\lambda_n (k) = A (k)$, that is, this process keeps track exactly of the number of active vertices. From this we obtain that if we define $T_N:=\min \{k:S^\lambda_n(k)=1-N\}$, then $T_N$ is the time at which the process has exhausted the exploration of the $N$th connected component (by order of exploration) $\cluster^{(N)}$ of $K_p(n,m)$. It follows that 

	\begin{equation}\label{eq:expsets}
	\Ver(\cluster^{(N)})=(\Dead (T_N)\setminus \Dead (T_{N-1}))\cup (\Opp (T_N)\setminus \Opp (T_{N-1})).
	\end{equation}
Moreover, if $v_0 \in \Ucal$,
	\begin{equation}\label{eq:expsides}\begin{split}
	\Ver(\cluster^{(N)})\cap \Ucal & = \Dead (T_N)\setminus \Dead (T_{N-1}),\\
	\Ver(\cluster^{(N)})\cap \Wcal &= \Opp (T_N)\setminus \Opp (T_{N-1}),
	\end{split}\end{equation}
 and vice versa if $v_0 \in \Wcal$.
These equalities allow us to reconstruct the scaling limit of the size of the largest components of $K_p(n,m)$ and of $G(n,m,p)$ from the process $\big(S^\lambda_n(k)\big)_{k\geq 0}$ as,  in particular, by Definition \ref{defi:intersection} we see that for every $v \in \Ucal$, 
	\begin{equation}\label{eq:rigcluster}
	\cluster_{G(n,m,p)}(v)=\cluster_{K_p(n,m)}(v) \cap \Ucal.
	\end{equation}

Considering how the RIG is defined, it would be reasonable to expect that the best approach would be to always start the exploration from $\Ucal$, since, doing so, the vertices explored in step $k$ are the actual neighbours of $v_k$ in $G(n,m,p)$. However, it turns out the the most effective proof strategy is to always start the exploration process from the smaller side, even when it is the  side which represents communities. This is the case because if the exploration starts from the smaller side all moments of $|\Ball_2(v_k)|$ are bounded and, consequently, the sequence of processes $(S_n^\lambda(k))_{k \geq 0}$ satisfies a central limit theorem as $n \to \infty$. If the exploration starts instead from the larger side, then $\Var (|\Ball_2(v_k)|)\to \infty$ and thus the exploration process is harder to study, since $\Ball_1(v_k)$ is typically empty, but if it is not, it is very large. We will prove most of the theorems first under the assumption that $\alpha >1$, so that exploration that is easier to study is also the more natural one, i.e., the one that starts from $\Ucal$. In the regime in which $\alpha<1$, in order to study the adapted process, we have to start the exploration from $\Wcal$ instead, and adapt all the results keeping in mind the following remark:

\begin{rk}[Symmetry properties of $K_p(n,m)$]\label{rk:invert}Note that since $K(n,m)$ is isomorphic to $K(m,n)$, all the statements proved using the exploration described in Definition \ref{def:exploration} and assuming $v_0 \in \Ucal$ also hold for the same exploration starting from $v_0 \in \Wcal$, with the roles of $n$ and $m$ and of $\Ucal$ and $\Wcal$ reversed. 
\end{rk}

\subsection{Proof of the main theorems, subject to the convergence of the adapted process to a Brownian motion}
\label{sec:brownian}

In this section we prove that the scaling limits in Theorems~\ref{thm:main} and \ref{thm:main2} hold, assuming the convergence of the rescaled adapted process to a Brownian motion in the $J_1$-topology (see e.g. \cite[Chapter 6]{JacShi87} for a detailed discussion about the properties of the $J_1$ topology).

The main tool for the proof of the scaling limit for component sizes expressed in Theorems~\ref{thm:main} and \ref{thm:main2} is a scaling limit for the process $\big(S^\lambda_n(k)\big)_{k\geq 0}$. Define the rescaled process

	\begin{equation}
	\overline{S}^\lambda_n (s) = n^{-1/3} S^\lambda_n\big( \lfloor s n^{2/3} \rfloor \big).
	\end{equation}
The following theorem establishes its limit in distribution:

\begin{thm}[Brownian limit of the adapted process]\label{the:brownian} Consider the two-step exploration process on $K_p(n,m)$  from Definition \ref{def:exploration}, starting from $v_0 \in \Ucal$, with $p= p_c (1+\lambda n^{-1/3})$ and $m=n^\alpha$, for some $\lambda \in \mathbb{R}, \alpha > 1$.  
Then, as $n\to \infty$, 
	\begin{equation}
	\big(\overline{S}^\lambda_n (s)\big)_{s\geq 0} \dto \big( W (s) +2\lambda s - s^2/2\big)_{s \geq 0},
	\end{equation}
where $\big( W (s)\big)_{s\geq 0}$ is a standard Brownian motion and the convergence in distribution is in the $J_1$-topology.
\end{thm}

%We compute the limit of the process in the case in which $\alpha > 1$ because it turns out to be easier, we will be able to deduct the results about the case $\alpha <1$ simply considering the size of the connected components in $\Wcal$ instead of $\Ucal$.

We next analyse the process $\big(R^\lambda_n (k)\big)_{k\geq 0}$, which is non-negative and whose excursions identify the explorations of the different connected components. We deduce from Theorem \ref{the:brownian} that

	\begin{equation}\begin{split}\label{eq:reflectconvergence}
	\big( \overline R^\lambda_n (s)\big)_{s\geq 0}&:=\big( n^{-1/3} R^\lambda_n\big( \lfloor s n^{2/3} \rfloor \big)\big)_{s\geq 0} \dto \big( \overline R^\lambda (s)\big)_{s \geq 0}, \\ 
	\overline R^\lambda (s) &:= W (s) +2\lambda s - s^2/2-\inf_{r\leq s}\big( W (r) +2\lambda r - r^2/2\big),
	\end{split}\end{equation} 
	since the reflection map is continuous \cite{DupIsh91} and $(W(s))_{s \geq 0}$ is continuous almost surely.
	
From now on, we will omit the rounding signs every time we approximate discrete-time variables by continuous ones, since such approximations never become an issue in the proofs.
	
We now show that  \eqref{eq:alfageq} in Theorem \ref{thm:main} follows from Theorem \ref{the:brownian}:

\begin{proof}[Proof of Theorem \ref{thm:main} (i) subject to Theorem \ref{the:brownian}] 

From \eqref{eq:expsides} and \eqref{eq:rigcluster} we obtain that 

	\begin{equation}\label{eq:deadtime}
	|\cluster^{(N)}\cap \Ucal| = D(T_N)-D(T_{N-1})=T_N-T_{N-1},
	\end{equation}
so the sizes of the clusters of $G(n,m,p)$ are exactly the lengths of the excursions of $\big(R^\lambda_n (k)\big)_{k\geq 0}$, that is, $n^{2/3}$ times the lengths of the excursions of $\big( \overline R^\lambda_n (s)\big)_{s \geq 0}$. From the convergence in distribution in \eqref{eq:reflectconvergence}, we know that for every $T,j >0$ the lengths of the $j$ longest excursions up to time $T$ of the process $\big(\overline R^\lambda_n (k)\big)_{k\geq 0}$ converge in distribution to those of the $j$ longest excursions of the process $\big(\overline R^\lambda (k)\big)_{k\geq 0}$ up to time $T$.
%We know from \cite{Ald97} that the excursions of the process $\big( \overline R^\lambda (s)\big)_{s \geq 0}$ can be ordered decreasingly almost surely and that they converge to a sequence of tight random variables, so for every $j \in \mathbb N$, the rescaled size of the $j$ largest components converges in distribution.
 Furthermore, by  \cite[Lemma 25]{Ald97}, for every $j\in \mathbb N$, $\varepsilon >0$, there exists a constant $T(j,\varepsilon)$ such that the $j$ longest excursions of $\big(\overline R^\lambda (k)\big)_{k\geq 0}$ are completed by time $T(j,\varepsilon)$ with probability larger than $1-\varepsilon$, so that the $j$ longest excursions of  $\big(\overline R_n^\lambda (k)\big)_{k\geq 0}$ converge in distribution to the $j$ longest excursions of $\big(\overline R^\lambda (k)\big)_{k\geq 0}$, so we obtain convergence in the product topology by \eqref{eq:deadtime}. 
Over finite dimensional spaces the product topology and the $\ell^2_\searrow$-topology are equivalent, so we know that for every fixed $J \in \mathbb N$, the sequence 
\begin{equation}
\big( n^{-2/3}|\cluster_i|\big)_{1\leq i\leq J}\dto \big(\mathbf{C}_i^{\lambda}\big)_{1\leq i\leq J}.
\end{equation}
 in the $\ell^2_\searrow$-topology. Moreover, by \eqref{eq:reflectconvergence}, we know that the sequence $\big( n^{-2/3}|\cluster_i|\big)_{i\geq 1}$ is tight in $\ell^2$ (see \cite{Ald97}), so that for every $\varepsilon \geq 0$

\begin{equation}
\lim_{J \to \infty} \lim_{n \to \infty} \P\Big( \sum_{i>J} n^{-4/3}|\cluster_i|^2 > \varepsilon\Big)=0,
\end{equation}
and the claim follows.
\end{proof}

In order to prove Theorem \ref{thm:main} (ii), a more complicated argument is needed, using Remark \ref{rk:invert}. Indeed we see that if $\alpha <1$, then for any vertex $v \in \Ucal$,

	\begin{equation}
	|\Ball_1(v) |\overset{d}{=} \Bin \big(n^\alpha,n^{-\frac{1+\alpha}{2}}(1+o(1))\big)\pto 0.
	\end{equation}
This means that the majority of the vertices on $\Ucal$ are isolated and thus $S_n^\lambda (k)=S_n^\lambda (k-1)-1$ w.h.p., and the process is thus driven by the rare events that happen  when $v_k$ is not isolated, as in this case its degree is very large, of order $\sqrt{ n/m}$. This makes the proof of a convergence analogous to that in Theorem \ref{the:brownian} much more difficult. To deal with this issue, we invert the perspective, running the exploration starting from a vertex $v_0 \in \Wcal$, so that we can apply again Theorem \ref{the:brownian}, since  $n=m^{1/\alpha}$, with $1/\alpha >1$. 
We recall that, if $v_0 \in \Wcal$,
%We notice that $K(n,m)$ is isomorphic to $K(m,n)$.
%We  define a graph $K(m,\overline{m})$ with $n=n$ and $m =m$ over the sets $\Ucal', \Wcal'$, as the graph obtained from an automorphism $\phi \in \mathbf{Aut} (K(n,m))$ such that $\phi:  K(n,m)\mapsto K(m,\overline{m})$   with $\phi(\Ucal')=\Wcal$, $\phi(\Wcal') = \Ucal$. 
%We see that 
%
%	\begin{equation}
%	\cluster_{G(n,m,p)}(v)=\cluster_{K_p(m,\overline{m})}(v) \cap \Wcal ',
%	\end{equation}
%	moreover if we run the exploration from Definition \ref{def:exploration} on $K_p(m,\overline{m})$, since $\overline{m}=m^{1/\alpha}$, with $1/\alpha >1$, now Theorem \ref{the:brownian} holds. We write
	
	\begin{equation}
\Ver( \cluster^{(N)})\cap \Ucal= \Opp (T_N)\setminus \Opp (T_{N-1}).
	\end{equation}
	Consequently, to prove Theorem \ref{thm:main} (ii), we need the following lemma, which we will prove in Section \ref{sec:mclt}:
	
	\begin{lem}[Concentration of size of the opposite set]\label{lem:concopp} Consider the two-step exploration process on $K_p(n,m)$ from Definition \ref{def:exploration}, starting from $v_0 \in \Ucal$, with $p= p_c (1+\lambda n^{-1/3})$ and $m=n^\alpha$ for some $\lambda \in \mathbb{R}, \alpha > 1$. Then as $n\to \infty$, for every $T\in (0, \infty)$,
	
	\begin{equation}
	\sup_{t\leq T}\Big| \dfrac{Q(tn^{2/3})}{n^{1/6}m^{1/2}}-t\Big| \pto 0.
	\end{equation}
	\end{lem}
We also need, in order to count the number of edges in large components and to prove convergence in the $\ell^2_\searrow$-topology, a uniform upper bound on the degree of vertices in $K_p(n,m)$.
\begin{lem}\label{lem:commsize} Consider the graph $K_p(n,m)$. Then, as $n\to \infty$, with $m=n^\alpha$ for a fixed $\alpha \in (0,1)$,
	\begin{equation}\label{eq:commsize1}
	\max_{w\in\Wcal} \partial B_1(w) \leq  np(1+\op(1)).
	\end{equation}
If instead $\alpha >1$,
	\begin{equation}\label{eq:commsize2}
	\max_{v\in\Ucal} \partial B_1(v) \leq  mp(1+\op(1)).
	\end{equation}
\end{lem}	

Further, in order to prove the scaling limit of the number of edges in the largest critical components, we need the following proposition, which we will prove in Section \ref{sec:edge}:

\begin{prop}[Concentration of the number of explored edges]\label{prop:surpll}
Consider the two-step exploration process on $K_p(n,m)$ from Definition \ref{def:exploration}, starting from $v_0 \in \Wcal$, with $p= p_c (1+\lambda m^{-1/3})$ and $m=n^\alpha$ for some $\lambda \in \mathbb{R}, \alpha < 1$.
 Recall the edge process $\big( E (k)\big)_{k \geq 1}=\big( |\mathcal E(k)|\big)_{k\geq 1}$ from Definition \ref{def:edge}. Then, for every  $T\in (0,\infty)$,

	\begin{equation}
	\sup_{t\leq T}\Big| \dfrac{E(tm^{2/3})m^{1/3}}{n}-\dfrac{t}{2}\Big| \pto 0.
	\end{equation}
\end{prop}	
	We can now give a proof of Theorem \ref{thm:main} (ii), assuming that the results we stated about the processes related to the exploration hold, which completes the proof of the main theorem.

	\begin{proof}[Proof of Theorem \ref{thm:main} (ii) subject to Theorem \ref{the:brownian}, Lemma \ref{lem:concopp} and Proposition \ref{prop:surpll}] Consider the two-step exploration process on $K_p(n,m)$ as defined in Definition \ref{def:exploration}, starting from $v_0\in \Wcal$. From Theorem \ref{the:brownian}, Lemma \ref{lem:concopp}  and Proposition \ref{prop:surpll} we obtain via Remark \ref{rk:invert}, by an application of Slutsky's theorem (see e.g. \cite[Lemma 3.3]{EthKur86}),
	
	\begin{equation}
	\big(\overline R_n^\lambda (s) , \overline{Q}_n(s), \overline{E}_n(s)\big)_{s\geq 0} \dto \big( \overline R^\lambda (s), s,s/2\big)_{s\geq 0},
	\end{equation}
	where \begin{equation}
	\overline{Q}_n(s):= Q(sm^{2/3})n^{-1/2}m^{-1/6}, \qquad \overline{E}_n(s):=E(sm^{2/3})n^{-1}m^{1/3} .
	\end{equation} 
	Consequently, the number of vertices found  in $\Ucal$ during the $N$th excursion is given by 
	\begin{equation}
	Q(T_N)-Q(T_{N-1})=(T_N-T_{N-1})m^{1/6}n^{1/2}(1+\op(1)),
	\end{equation}  
and, equivalently, the number of edges discovered during the exploration on the $N$th component is given by 
	\begin{equation}
	 E(T_N)-E(T_{N-1})=(T_N-T_{N-1})m^{-1/3}n(1+\op(1)).
	\end{equation}
	Again,  by the same reasoning used to prove Theorem \ref{thm:main} (i), from the convergence in the $J_1$-topology follows that the $j$ longest excursions of $\big( \overline R_n^\lambda (s)\big)_{s\geq 0}$ converge in distribution to the $j$ longest excursions of $\big( \overline R^\lambda (s)\big)_{s\geq 0}$, and so, the convergence holds in the product topology. 	
Again, over finite dimensional spaces the product topology is equivalent to the $\ell^2$-topology, so

\begin{equation}
	\big( n^{-1/2-\alpha/6}|\cluster_i|, n^{-1+\alpha/3}|\Ed(\cluster_i)|\big)_{1\leq i\leq J} \dto \big(\mathbf{C}_i^{\lambda},\mathbf{C}_i^{\lambda}/2\big)_{1\leq i\leq J}.
	\end{equation} 
 in the $\ell^2_\searrow\times \ell^2$-topology. Consequently, to prove the convergence in the $\ell^2_\searrow\times \ell^2$ we need to prove that for every $\varepsilon >0$,
\begin{equation}
\lim_{J \to \infty} \lim_{n \to \infty} \P\Big( \sum_{i>J}( n^{-1-\alpha/3}|\cluster_i|^2+ n^{-2+2\alpha/3}|\Ed(\cluster_i)|^2) > \varepsilon\Big)=0.
\end{equation}
We define the sequence $(\tilde T_i)_{i \geq 1}$ as the sequence of exploration times of  connected components, arranged in decreasing order.
By \eqref{eq:reflectconvergence}, we know that the sequence $(m^{-2/3}\tilde T_i)_{i \geq 1}$ is tight in $\ell^2$, so 

\begin{equation}\label{eq:tailsmall}
\lim_{J \to \infty} \lim_{n \to \infty} \P\Big( \sum_{i>J} m^{-4/3} \tilde T_i^2 > \varepsilon/8\Big)=0.
\end{equation}
We know that for every $i$,

\begin{align}
|\cluster_i^\lambda|&\leq \tilde T_i\max_{w \in \Wcal} \partial B_1(w) , \\
|\Ed(\cluster_i)|&\leq  \tilde T_i \max_{w \in \Wcal}\frac{\partial B_1(w)(\partial B_1(w)-1)}{2},
\end{align}
since all the new vertices and edges explored in the $k$th step must belong to the community formed by all the neighbours of $w_k$ in $K_p(n,m)$ .
We can thus write
\begin{equation}\begin{split}
\P\Big( \sum_{i>J}& (n^{-1-\alpha/3}|\cluster_i|^2+ n^{-2+2\alpha/3}|\Ed(\cluster_i)|^2) > \varepsilon\Big)\\ 
&\leq \P\Big( \sum_{i>J} m^{-4/3}\tilde T_i^2 > \varepsilon/8\Big)+ \P\Big( \max_{w \in \Wcal} \partial B_1(w) > 2pn\Big).
\end{split}\end{equation}
The claim follows from Lemma \ref{lem:commsize} and \eqref{eq:tailsmall}\end{proof}

	We next analyse the critical behaviour of $K_p(n,m)$, and prove Theorem \ref{thm:main2}:
	
%	For this we need the following proposition on the number of surplus edges found by the exploration:
%	
%	\begin{prop}[Scaling of the surplus of the bipartite graph]\label{prop:surppoi}
%Consider the two-steps exploration process on $K_p(n,m)$ from Definition \ref{def:exploration}, starting from $v_0 \in \Ucal$, with $p= p_c (1+\lambda n^{-1/3})$ and $m=n^\alpha$ for some $\lambda \in \mathbb{R}, \alpha > 1$.
% Recall $\big(\mathbf{Sp}(k)\big)_{k>0}$ from Definition \ref{def:edge}, and define $(J(s))_{s\geq 0}$ 
%as the unit jump process  with intensity $(\overline R^\lambda(s))_{s\geq 0}$. Then, as $n \to \infty$,
%	\begin{equation}	
%	\big(\mathbf{Sp}( sn^{2/3})\big)_{s\geq 0}\dto \big(J(s)\big)_{s\geq 0},
%	\end{equation}
%where the convergence is in finite-dimensional distributions.
%\end{prop}
%We can now give a conditional proof of Theorem \ref{thm:main2}:

\begin{proof}[Proof of Theorem \ref{thm:main2} subject to Theorem \ref{the:brownian} and Lemma \ref{lem:concopp}] Using Theorem \ref{the:brownian} and Lemma \ref{lem:concopp} we obtain by Slutsky's theorem the joint convergence

	\begin{equation}\label{eq:surpconv}
	\big(\overline R_n^\lambda (s) , \overline{Q}_n(s)\big)_{s\geq 0} \dto \big( \overline R^\lambda (s), s\big)_{s\geq 0}.
	\end{equation}
We know that the number of vertices in the $N$th component by exploration order is given by
	\begin{equation}
 	Q(T_N)-Q(T_{N-1})+D(T_N)-D(T_{N-1})=Q(T_N)-Q(T_{N-1})+T_N-T_{N-1}.
	\end{equation}
From \eqref{eq:surpconv}, it follows that
\begin{equation}
 Q(T_N)-Q(T_{N-1})+T_N-T_{N-1}=m^{1/2}n^{1/6}(T_N-T_{N-1})(1+\op(1)),
\end{equation}
in particular, the vast majority of the vertices in a large critical component comes from the larger side of $K_p(n,m)$, i.e. from the set $\Opp (T_N)\setminus\Opp (T_{N-1})$.
	Again,  by the same reasoning used to prove Theorem \ref{thm:main}, from the convergence in the $J_1$-topology follows that the $j$ longest excursions of $\big( \overline R_n^\lambda (s)\big)_{s\geq 0}$ converge in distribution to the $j$ longest excursions of $\big( \overline R^\lambda (s)\big)_{s\geq 0}$, (see \cite{Ald97}), and so, the claim follows. The proof of the $\ell^2_\searrow$ in \eqref{eq:bipa} is very similar to the one we just presented for \eqref{eq:alfamin}, minus the argument for the $\ell^2$ convergence of the number of edges. As before, the convergence in the product topology is enough to obtain, for every $J \in \mathbb N$,

\begin{equation}
\big( n^{-1/6-\alpha/2}|\cluster_i|\big)_{1\leq i\leq J} \dto \big(\mathbf{C}_i^{\lambda}\big)_{1\leq i\leq J},
\end{equation}
in the $\ell^2_\searrow$-topology. Again, by \eqref{eq:reflectconvergence}, we know that the sequence $(m^{-2/3}\tilde T_i)_{i \geq 1}$ is tight in $\ell^2$, so 

\begin{equation}\label{eq:tailsmall2}
\lim_{J \to \infty} \lim_{n \to \infty} \P\Big( \sum_{i>J} m^{-4/3} \tilde T_i^2 > \varepsilon/4\Big)=0.
\end{equation}
At every step $k$ of the exploration, the vertices whose exploration is completed are $v_k$ and all its neighbours, so that

\begin{equation}
|\cluster_i|\leq \tilde T_i\Big(1+\max_{w \in \Wcal} \partial B_1(w)\Big).
\end{equation}
Consequently, we can write
\begin{equation}\begin{split}
\P\Big( \sum_{i>J}& n^{-1/3-\alpha}|\cluster_i|^2 > \varepsilon\Big)\\ 
&\leq \P\Big( \sum_{i>J} m^{-4/3} \tilde T_i^2 > \varepsilon/4\Big)+ \P\Big( \max_{w \in \Wcal} \partial B_1(w) > 2pm-1\Big).
\end{split}\end{equation}
Using Lemma \ref{lem:commsize} and \eqref{eq:tailsmall2}, we obtain that
\begin{equation}
\lim_{n \to \infty}\P\Big( \sum_{i>J} n^{-1/3-\alpha}|\cluster_i|^2 > \varepsilon\Big)=0
\end{equation}
and thus the claimed convergence in the $\ell^2_\searrow$-topology follows.
\end{proof}

	\section{The martingale central limit theorem}
	\label{sec:mclt}

In this section we proceed to prove Theorem \ref{the:brownian}. We follow the approach proposed in \cite{Ald97} which has already inspired many other papers (e.g., \cite{BhaHofLee12,DhaHofLeeSen16b,FedHofHolHul16a,NacPer10}), based on the Martingale Functional Central Limit Theorem (MFCLT). This is a well-established proof technique, we next recall its main elements to make the paper self contained.
We will closely follow the formulation of the MFCLT given in \cite{FedHofHolHul16a}, with the necessary modifications required to adapt it to the exploration of the graph $K_p(n,m)$, whose geometry is quite different from the one studied in \cite{FedHofHolHul16a}.
The MFCLT can be applied to a discrete-time process $\big( S_n (k)\big)_{k\geq 0}$ which admits  a Doob's decomposition as follows (see \cite{Will91} for example). We define the process $\big(M_n(k)\big)_{k\geq 0}$ as a martingale,  $\big(L_n(k)\big)_{k\geq 0}$ as its quadratic variation process, and $\big(\Fil_n(k)\big)_{k \geq 0}$ as a filtration with respect to which $\big( S_n (k)\big)_{k\geq 0}$ is measurable. We then decompose $\big( S_n (k)\big)_{k\geq 0}$ as

	\begin{equation}
	S_n (k) := Y_n(k)+M_n(k), \qquad M_n(k)^2:=Q_n(k)+L_n(k),
	\end{equation}
where
	\begin{equation}\label{eq:doob}\begin{split}
	Y_n(k)&:=\sum_{j=1}^k \E[S_n (j)-S_n (j-1)\mid \Fil_n(j-1)], \qquad\\ L_n(k)&:=\sum_{j=1}^k \Var(S_n (j)-S_n (j-1)\mid \Fil_n(j-1)).
	\end{split}\end{equation}
We can now state the conditions  on $\big(S_n (k)\big)_{k\geq 0}$ required to apply the MFCLT, in a way that is easy to apply to our process:
\begin{thm}[MFCLT] 
\label{thm:MFCLTcond}
Consider a sequence of processes $(S_n(k))_{k \geq 0}$, adapted to a sequence of increasing 
filtrations $\Fil_n=\big(\Fil_n(k)\big)_{k \geq 0}$. Suppose that $\big(M_n(k)\big)_{k\geq 0},
\big(Y_n(k)\big)_{k\geq 0}$, and $\big(L_n(k)\big)_{k\geq 0}$, defined as in the Doob decomposition 
above, satisfy the following four conditions, for a fixed $\tau \in \mathbb R$ and every $t\in (0, \infty)$:
\begin{enumerate}
\item 
Continuity of the limit of the process: 
	\begin{equation}
	\label{cond(1)}
	n^{-2/3}\,\E\Big[\sup_{k \leq t n^{2/3}} |M_n (k) -M_n(k-1)|^2 \Big] \to 0.
	\end{equation}
\item
Continuity of the limit of the variance:
	\begin{equation} 
	\label{cond(2)}
	n^{-2/3}\,\E\Big[\sup_{k \leq tn^{2/3}} |L_n (k) - L_n(k-1)| \Big] \to 0.
	\end{equation}
\item 
Parabolic drift condition:
	\begin{equation}
	\label{cond(4)}
	n^{-1/3}\, \sup_{k \leq tn^{2/3}} \left| Y_n(k) - \frac{2\tau k}{ n^{1/3}} 
	+ \dfrac{k^2}{2n} \right| \pto 0.
	\end{equation}
\item 
Limiting linear variance condition:
	\begin{equation}
	\label{cond(3)}
	n^{-2/3}\,L_n( tn^{2/3}) \pto t.
	\end{equation}
\end{enumerate}
Then
	\begin{equation}
	\big(n^{-1/3}S_n( sn^{2/3})\big)_{s \geq 0}  \dto 
	\left(W(s) + 2\tau s - \frac{s^2}{2}\right)_{s \geq 0},
	\end{equation}
	 in the Skorokhod $J_1$-topology, where $\big(W(s)\big)_{s\geq 0}$ is a standard Brownian motion.
\end{thm}
	
The main goal of the rest of this section will be to prove that conditions $(1)-(4)$ apply for $	(S_n(k))_{k \geq 0}=(S_n^\lambda(k))_{k \geq 0}$ as defined in \eqref{eq:adapted}. The application of the MFCLT will yield the proof of Theorem \ref{the:brownian}.
	
\subsection{Continuity conditions for the MFCLT}
\label{sec:continuity}
In the following lemma we prove that the first two conditions hold for the sequence of processes $\big(S_n^\lambda (k)\big)_{k\geq 0}$, as defined in \eqref{eq:adapted}:

\begin{lem}[Continuity of the limit process]\label{lem:12}
Consider the Doob's decomposition of the process $\big(S_n^\lambda (k)\big)_{k\geq 0}$ adapted to the two-step exploration process from Definition \ref{def:exploration}, starting from $v_0 \in \Ucal$, with $p= p_c (1+\lambda n^{-1/3})$ and $m=n^\alpha$ for some $\lambda \in \mathbb{R}, \alpha > 1$. Then, as $n\to \infty$, for every $t \in (0,\infty)$,

\begin{equation}
	\label{lem:cond(1)}
	n^{-2/3}\,\E\Big[\sup_{k \leq t n^{2/3}} \big|M_n (k) -M_n(k-1)\big|^2 \Big] \to 0,
	\end{equation}
	\begin{equation} 
	\label{lem:cond(2)}
	n^{-2/3}\,\E\Big[\sup_{k \leq tn^{2/3}} \big|L_n (k) - L_n(k-1)\big| \Big] \to 0.
	\end{equation}
	\proof
	We separately consider the two cases in which the supremum of $ |M_n (k) -M_n(k-1)|$ is achieved in the positive or negative part:
\begin{equation}\begin{split}
	\E\Big[\sup_{k \leq t n^{2/3}} |M_n (k) -M_n(k-1)|^2 \Big]&= \E\Big[\max \Big\{\sup_{k \leq t n^{2/3}} \big((M_n (k) -M_n(k-1)\big)^+)^2, \\ &\qquad \sup_{k \leq t n^{2/3}} \big((M_n (k) -M_n(k-1)\big)^-)^2 \Big\}\Big].
	\end{split}\end{equation}	
We thus obtain

	\begin{equation}\begin{split}\label{eq:supsplit} 
	\E\Big[\sup_{k \leq t n^{2/3}} |M_n (k) -M_n(k-1)|^2 \Big]&\leq \E\Big[ \sup_{k \leq t n^{2/3}} \big((M_n (k) -M_n(k-1)\big)^+)^2\Big] \\ &\qquad +\E\Big[ \sup_{k \leq t n^{2/3}} \big((M_n (k) -M_n(k-1)\big)^-)^2 \Big].
	\end{split}\end{equation}
	
	We define the random variables
	
	\begin{equation}
	Z_1:=\Bin (m,p), \qquad Z_2:=\Bin (nZ_1,p).
	\end{equation}
	We note that, for every $v \in \Ucal$, 
	\begin{equation}\label{eq:balldom}
	Z_1\overset{d}{=}|\Ball_1(v)|,\quad Z_2\succeq |\Ball_2(v)|,
	\end{equation}
since every element in $\Ball_1(v)$ has at most $n-1$ other potential neighbours in $\Ucal$. Moreover, $\Ball_2(v)\supseteq \Ball_2 (v,\Opp (k-1)\cup \Dex (k-1))$, since $2$ is the minimal distance between vertices in $\Ucal$, so that, for every $k$, $X_n(k) \preceq Z_2-1$. By definition, $X_n(k)\geq -1$ deterministically for all $k$. We write
	
	\begin{equation}\label{eq:martdef}
	M_n(k)-M_n(k-1)=X_n(k)-\E[X_n(k)\mid \Fil_n(k-1)],
	\end{equation}
we obtain thus that almost surely
	\begin{equation}
	M_n(k)-M_n(k-1)\geq -1-\E[Z_2 -1] = -\E[Z_2 ],
	\end{equation}
	so that
	\begin{equation}
	\E\Big[ \sup_{k \leq t n^{2/3}} \big((M_n (k) -M_n(k-1)\big)^-)^2\Big]\leq \E[Z_2]^2.
	\end{equation}
We compute
	\begin{equation}
	\E[Z_2]=pn \E[Z_1]=p^2nm=1+o(1),
	\end{equation}
	and thus
	\begin{equation}\label{eq:negpart}
	\E\Big[ \sup_{k \leq t n^{2/3}} \big((M_n (k) -M_n(k-1)\big)^-)^2\Big]\leq 1+o(1).
	\end{equation}
On the other hand, we see that, since 
\begin{equation}
(X_n(k)\mid \Fil_n(k-1))\preceq Z_2-1,\quad\E[X_n(k)\mid \Fil_n(k-1)]\geq-1, \quad a.s,
\end{equation}
we obtain, by \eqref{eq:martdef}
\begin{equation}
\big(M_n(k)-M_n(k-1)\mid \Fil_n(k-1)\big)\preceq Z_2-1+1=Z_2, \quad a.s.,
\end{equation}
and consequently, since $Z_2$ is a.s. non-negative,
\begin{equation}
\big((M_n(k)-M_n(k-1))^+\mid \Fil_n(k-1)\big)\preceq Z_2,\quad a.s..
\end{equation}
Note that here and in the rest of the paper we treat the conditional distributions as random variables which take values in the space of real-valued random variables, which is naturally equipped with the partial order given by stochastic domination.
  Thus, we can stochastically dominate the sequence
	$\big((M_n(k)-M_n(k-1))^+\big)_{k \geq 1}$ with a sequence $(Z_2(k))_{k\geq 1}$ of i.i.d. random variables with the same distribution as $Z_2$.
We thus obtain from \eqref{eq:supsplit} and \eqref{eq:negpart} that

	\begin{equation}\begin{split}\label{eq:twotermsz}
	\E\Big[\sup_{k \leq t n^{2/3}} |M_n (k) -M_n(k-1)|^2 \Big]\leq \E\Big[\sup_{k \leq t n^{2/3}} Z_2(k)^2\Big]+O(1).
	\end{split}\end{equation}
 For \eqref{lem:cond(1)}, we thus need to prove that

	\begin{equation}\label{eq:maxz}
	\E\Big[\sup_{k \leq t n^{2/3}} Z_2(k)^2\Big]=o(n^{2/3}).
	\end{equation}
We define for any $\varepsilon, k >0$, the events

	\begin{equation}
	\Ical_\varepsilon(k):= \{ Z_2(k) \geq \varepsilon n^{1/3}\}, \quad \Ical_\varepsilon:= \Big\{\bigcup_{k\leq tn^{2/3}} \Ical_\varepsilon(k)\Big\}.
	\end{equation}
 We write
	\begin{equation}\label{eq:isplit}
	\E\Big[\sup_{k \leq t n^{2/3}} Z_2(k)^2\Big]=\E\Big[\sup_{k \leq t n^{2/3}} Z_2(k)^2\indi_{\Ical_\varepsilon^c}\Big]+\E\Big[\sup_{k \leq t n^{2/3}} Z_2(k)^2\indi_{\Ical_\varepsilon}\Big].
	\end{equation}
By definition of $\Ical_\varepsilon$,
	\begin{equation}\label{eq:eni}
	\E\Big[\sup_{k \leq t n^{2/3}} Z_2(k)^2\indi_{\Ical_\varepsilon^c}\Big]\leq n^{2/3}\varepsilon^2.
	\end{equation}
To bound the second term, we write,

	\begin{equation}\label{eq:sumsup}\begin{split}
	\E\Big[\sup_{k \leq t n^{2/3}} Z_2(k)^2\indi_{\Ical_\varepsilon}\Big]\leq \E\Big[ \sum_{k \leq t n^{2/3}} Z_2(k)^2\indi_{\Ical_\varepsilon(k)}\Big]=tn^{2/3}\E[ Z_2(k)^2\indi_{\Ical_\varepsilon(k)}],
	\end{split}\end{equation}
where the last term actually does not depend on $k$ due to the i.i.d. nature of the variables $Z_2(k)$.
Thus, we compute, for every $k$,

	\begin{align}
	\E[Z_2(k)^4]&\geq \E[Z_2(k)^4\indi_{\Ical_\varepsilon (k)}]= \E[Z_2(k)^4\mid \Ical_\varepsilon (k)]\P(\Ical_\varepsilon (k))
	\\&\geq \E[Z_2(k)^2\mid \Ical_\varepsilon (k)]^2\P(\Ical_\varepsilon (k)) =\E[Z_2(k)^2\indi_{\Ical_\varepsilon (k)}]\E[Z_2(k)^2\mid \Ical_\varepsilon (k)],\nonumber
	\end{align}
so that
	\begin{equation}\label{eq:mom4}
	\E[Z_2(k)^2\indi_{\Ical_\varepsilon (k)}]\leq \dfrac{\E[Z_2(k)^4]}{\E[Z_2(k)^2\mid \Ical_\varepsilon (k)]}\leq \dfrac{\E[Z_2(k)^4]}{n^{2/3}\varepsilon^2}=\dfrac{\E[Z_2^4]}{n^{2/3}\varepsilon^2}.
	\end{equation}
Thus, we need an upper bound on the fourth moment of $Z_2$.
We recall that $Z_2$ can be written as 

	\begin{equation}
	Z_2=\sum_{i=1}^{Z_1} Y_i, \quad (Y_i)_{i=1}^{Z_1} \text{ i.i.d. with distribution }\Bin(n,p),
	\end{equation}
or, equivalently, since $Z_1\overset{d}{=}\Bin (m,p)$,

	\begin{equation}
	Z_2= \sum_{i=1}^{m} Y_iW_i,  \quad  (W_i)_{i=1}^m\text{ i.i.d. with distribution } \Ber(p),
	\end{equation}
where the sequences $(Y_i)_{i=1}^{m}$ and $(W_i)_{i=1}^m$ are independent of each other.
Thus, we can write,
	\begin{equation}
	\E[Z_2^4]=\sum_{i_1,i_2,i_3,i_4\in [m]^4} \E[W_{i_1}W_{i_2}W_{i_3}W_{i_4}]\E[Y_{i_1}Y_{i_2}Y_{i_3}Y_{i_4}].
	\end{equation}
We divide the sum in five different terms, based on how many of the indices coincide. We analyse case by case:
	\begin{equation}\begin{split}
	\E[Z_2^4]\leq&\Big(\sum_{|\{i_1,i_2,i_3,i_4\}|=4}p^4\E[Y_{i_1}]\E[Y_{i_2}]\E[Y_{i_3}]\E[Y_{i_4}]+ 6\sum_{|\{i_1,i_2,i_3\}|=3} p^3\E[Y_{i_1}^2]\E[Y_{i_2}]\E[Y_{i_3}]\\
	 \quad &+\sum_{i_1\neq i_2} p^2(3\E[Y_{i_1}^2]\E[Y_{i_2}^2]+4\E[Y_{i_1}^3]\E[Y_{i_2}])+\sum_{i_1} p\E[Y_{i_1}^4]\Big)(1+o(1)).
	\end{split}\end{equation}
We use that $mnp^2=1+o(1)$ and that, when $np\to 0$, $\E[\Bin(n,p)^j]=np(1+o(1))$ for all $j\geq 1$, to obtain

	\begin{equation}\label{eq:15}
	\E[Z_2^4]=\Big((mnp^2)^4+6(mnp^2)^3+7(mnp^2)^2+mnp^2\Big)(1+o(1))=15+o(1).
	\end{equation}

We can now substitute \eqref{eq:15} into \eqref{eq:mom4} and then into \eqref{eq:sumsup} to obtain

	\begin{equation}\label{eq:ei}
	\E\Big[\sup_{k \leq t n^{2/3}} Z_2(k)^2\indi_{\Ical_\varepsilon}\Big]\leq tn^{2/3}\dfrac{15(1+o(1))}{n^{2/3}\varepsilon^2}=15t\varepsilon^{-2} +o(1)=O(1).
	\end{equation}
Substituting \eqref{eq:eni} and \eqref{eq:ei} into \eqref{eq:isplit} we prove \eqref{eq:maxz}.
From \eqref{eq:maxz} and \eqref{eq:twotermsz} we finally obtain \eqref{lem:cond(1)}. 

To prove \eqref{lem:cond(2)} we use that

	\begin{equation}
	L_n(k)-L_n(k-1)=\Var (X_n(k)\mid \Fil_n(k-1))=\Var (X_n(k)+1\mid \Fil_n(k-1)).
	\end{equation}
Since, for all $k$, $0\leq (X_n(k)+1\mid \Fil_n(k-1)))\preceq Z_2$ almost surely, we obtain that

	\begin{equation}
	\Var(X_n(k)\mid \Fil_n(k-1))\leq \E[X_n(k)^2\mid \Fil_n(k-1)] \leq \E[Z_2^2].
	\end{equation}
	and consequently
	\begin{equation}
	n^{-2/3}\,\E\Big[\sup_{k \leq tn^{2/3}} |L_n (k) - L_n(k-1)| \Big]\leq n^{-2/3} \E[Z_2^2].
	\end{equation}
We write
	\begin{equation}
	\E[Z_2^2]\leq \E[Z_2^4]^{1/2}=O(1),
	\end{equation}
using \eqref{eq:15}, so that \eqref{lem:cond(2)} follows. \qed
\end{lem}

\subsection{Parabolic drift condition}
\label{sec:parabolic}
Before proving the parabolic drift condition and the linear variance condition we now prove Lemma \ref{lem:concopp}, which, other than being necessary to derive from Theorem \ref{the:brownian} results relative to the case $\alpha <1$, is also relevant to the proof of the last two conditions required for the application of the MFCLT.	

\begin{proof}[Proof of Lemma \ref{lem:concopp}]
We start the proof with an upper bound on $Q(k)$. By Definition \ref{def:exploration},

\begin{equation}\label{eq:distopp}
\big( Q(k)-Q(k-1)\mid \Fil_n(k-1)\big)\overset{d}{=}\Bin (m-Q(k-1),p).
\end{equation} 
Consequently, for each $k$, almost surely

	\begin{equation}
	\big( Q(k)-Q(k-1)\mid \Fil_n(k-1)\big)\preceq Z_1,
	\end{equation}
so that 
	\begin{equation}
	Q(k)\preceq \sum_{j=1}^k  Z_1(j),
	\end{equation}
where $(Z_1(j))_{j\geq 1}$ are i.i.d. random variables with the same distribution as $Z_1$. We thus obtain
	\begin{equation}\label{eq:upperp}
	\P\Big(\dfrac{Q(tn^{2/3})}{m^{1/2}n^{1/6}}-t>\varepsilon\Big)\leq \P\Big(\dfrac{\sum_{j=1}^{tn^{2/3}}  Z_1(j)}{ m^{1/2}n^{1/6}}-t>\varepsilon\Big).
	\end{equation}
We compute
	\begin{equation}\label{eq:expP}
	\E\Big[\sum_{j=1}^k  Z_1(j)\Big]=k\E[Z_1]=kpm=k\sqrt{\dfrac{m}{n}}(1+o(1)),
	\end{equation}
while 
	\begin{equation}\label{eq:varP}
	\Var\Big(\sum_{j=1}^k  Z_1(j)\Big) =k \Var(Z_1)=kp(1-p)m=k\sqrt{\dfrac{m}{n}}(1+o(1)).
	\end{equation}
	From \eqref{eq:expP} we know that for every $t, \varepsilon >0$, there is $\overline n$ large enough that
	\begin{equation}
	\dfrac{\E\big[\sum_{j=1}^{tn^{2/3}}  Z_1(j)\big]}{m^{1/2}n^{1/6}}-t<\varepsilon/2, \quad \forall n > \overline n.
	\end{equation}
	Consequently, 
		\begin{equation}\begin{split}
	\limsup_{n \to \infty}\P&\Big(\dfrac{\sum_{j=1}^{tn^{2/3}}  Z_1(j)}{m^{1/2}n^{1/6}}-t>\varepsilon\Big)\\&\leq \limsup_{n \to \infty}\P\Big(\dfrac{\sum_{j=1}^{tn^{2/3}}  Z_1(j)- \E\big[\sum_{j=1}^{tn^{2/3}}  Z_1(j)\big]}{m^{1/2}n^{1/6}}>\varepsilon/2\Big).
	\end{split}\end{equation}
Thus, by the second moment method, we obtain from \eqref{eq:upperp} that
	\begin{equation}\label{eq:upopp}
	\P\Big(\dfrac{Q(tn^{2/3})}{m^{1/2}n^{1/6}}-t>\varepsilon\Big)\leq \P\Big(\dfrac{\sum_{j=1}^{tn^{2/3}}  Z_1(j)}{ m^{1/2}n^{1/6}}-t>\varepsilon\Big) 	\to 0.
	\end{equation}
Now we prove the matching lower bound. We define the random variable $\widetilde Z_1 \overset{d}{=} \Bin (m-(t+\varepsilon) m^{1/2}n^{1/6},p))$ and note that, by \eqref{eq:distopp}, for all $k \leq tn^{2/3}$,
	\begin{equation}
	(Q(k)-Q(k-1)\mid \Fil_n(k-1))\succeq \widetilde Z_1,
	\end{equation}
when
	\begin{equation}\label{eq:uniupopp}
	Q(tn^{2/3}-1) \leq (t+\varepsilon) m^{1/2}n^{1/6}.
	\end{equation}
 By  \eqref{eq:upopp}, we obtain that \eqref{eq:uniupopp} is satisfied w.h.p.. Consequently, there exists a coupling between $(Q(j))_{j \geq 1}$ and a sequence of i.i.d. random variables $(\widetilde Z_1(j))_{j\geq 1}$ with the same distribution as $\widetilde Z_1$ such that,
 \begin{equation}
 \lim_{n \to \infty} \P(\exists j< tn^{2/3}: Q(j)-Q(j-1)<\widetilde Z_1(j))=0.
 \end{equation}
We thus deduce that
	\begin{equation}
	\limsup_{n\to \infty}\P\Big(\dfrac{Q(tn^{2/3})}{m^{1/2}n^{1/6}}-t<-\varepsilon\Big)\leq \limsup_{n\to \infty}\P\Big(\dfrac{\sum_{j=1}^{tn^{2/3}}  \widetilde Z_1(j)}{ m^{1/2}n^{1/6}}-t<-\varepsilon\Big).
	\end{equation}
%
%	\begin{equation}
%	Q(k)\succeq \sum_{j=1}^k \widetilde Z_1(j),
%	\end{equation}
We compute
	\begin{equation}
	\E\Big[\sum_{j=1}^k \widetilde Z_1(j)\Big]=k(m-(t+\varepsilon) m^{1/2}n^{1/6})p=k\sqrt{\dfrac{m}{n}}(1+o(1)), 
\end{equation}
and
\begin{equation}
 \Var\Big(\sum_{j=1}^k \widetilde Z_1(j)\Big) =k(m-(t+\varepsilon) m^{1/2}n^{1/6})p(1-p)= k\sqrt{\dfrac{m}{n}}(1+o(1)).
\end{equation}
We can thus prove the matching lower bound to \eqref{eq:upopp}, again by the second moment method:
	\begin{equation}
	\limsup_{n\to \infty}\P\Big(\dfrac{Q(tn^{2/3})}{m^{1/2}n^{1/6}}-t<-\varepsilon\Big)\leq \limsup_{n\to \infty}\P\Big(\dfrac{\sum_{j=1}^{tn^{2/3}} \widetilde Z_1(j)}{ m^{1/2}n^{1/6}}-t<-\varepsilon\Big)= 0.
	\end{equation}
This proves pointwise convergence of the process $\Big(\dfrac{Q(tn^{2/3})}{m^{1/2}n^{1/6}}\Big)_{t\geq 0}$, to the identity function. The uniform convergence follows in a standard way from the fact that $(Q(tn^{2/3}))_{t \geq 0}$ is non-decreasing in $t$ and that the identity function is continuous.
\end{proof}

We also prove an upper bound on the size of $\Act(k)$, which will be used in the proofs of the last two conditions to control the effect of the depletion of points on the distribution of $X_n(k)$:

\begin{lem}[Upper bound on the size of the active set]\label{lem:supactive} Consider the two-step exploration process on $K_p(n,m)$, from Definition \ref{def:exploration}, starting from $v_0 \in \Ucal$, with $p= p_c (1+\lambda n^{-1/3})$ and $m=n^\alpha$ for some $\lambda \in \mathbb{R}, \alpha > 1$. Then, as $n\to \infty$,  for every $t\in (0,\infty)$,
	
	\begin{equation}
	n^{-2/3}\sup_{k\leq tn^{2/3}} A(k) \pto 0.
	\end{equation}

\end{lem}
\proof We recall that
	\begin{equation}
	 A(k)= R_n^\lambda(k)=S_n^\lambda(k)-\min_{j< k}S_n^\lambda(j)+1.
	\end{equation}
From this we obtain
	\begin{equation}\label{eq:supactive}
	\Big\{ \sup_{k\leq tn^{2/3}} A(k) < \varepsilon n^{2/3} \Big\} \supseteq \big\{-\varepsilon n^{2/3}/3 < S_n^\lambda(k) < \varepsilon n^{2/3} /3 \ \forall k \leq tn^{2/3} \big\}.
	\end{equation}
For the upper bound, we already argued from \eqref{eq:balldom} that $(S_n^\lambda(k)-S_n^\lambda(k-1)\mid \Fil_n(k-1))\preceq Z_2-1$ for all $k$ a.s., and consequently,

	\begin{equation}
S_n^\lambda(k) \preceq \sum_{j=1}^k (Z_2(j)-1),
	\end{equation}
where $(Z_2(j))_{j \geq 1}$ are i.i.d. random variables with distribution identical to $Z_2$. 
We now bound, uniformly for all $k\leq tn^{2/3}$,
	\begin{equation}
	\E\Big[\sum_{j=1}^k( Z_2(j)-1)\Big] =  2k \lambda n^{-1/3}(1+o(1))= \Theta (n^{1/3}),
	\end{equation}
	and
	\begin{equation}
	\Var\Big(\sum_{j=1}^k( Z_2(j)-1)\Big)\leq k \E[Z_2^2]= k O(1)=O(n^{2/3}).
	\end{equation}
From this we conclude that, by Kolmogorov's inequality (see e.g. \cite[Theorem 7.8.2]{GriSti92}),
	\begin{equation}\label{eq:actupp}\begin{split}
	\P\Big(\max_{k \leq tn^{2/3}}S_n^\lambda(k)> \varepsilon n^{2/3}/3\Big)&\leq \P \Big(\max_{k \leq tn^{2/3}}\sum_{j=1}^k (Z_2(j)-1)> \varepsilon n^{2/3}/3\Big)\\ &\leq \frac{\Var\Big(\sum_{j=1}^{tn^{2/3}}
	 (Z_2(j)-1)\Big)}{\Big(\varepsilon n^{2/3}/3-\max_{k \leq tn^{2/3}}\sum_{j=1}^k \E[(Z_2(j)-1)]\Big)^2}\\
&=\frac{O(n^{2/3})}{\varepsilon^2 n^{4/3} (1-o(1))/9}\to 0.
	\end{split}\end{equation}

For the matching lower bound, we note that 
	\begin{equation}\label{eq:doulbebin}
	X_n(k) \overset{d}{=}\Bin \Big((n-D^+(k-1))\Bin \big(m-Q(k-1),p\big),p\Big),
	\end{equation}
	since $v^k$ has $m-Q(k-1)$ potential neighbours available to be explored in $\Wcal$ and each of them in turn has $n-D^+(k-1)$ potential neighbours in $\Ucal$ and all edges are present with probability $p$ independently of each other.
We proved in Lemma \ref{lem:concopp} that $\max_{k \leq tn^{2/3}} Q(k)=\Tp(m^{1/2}n^{1/6})$. Moreover, by Definition \ref{def:exploration}, $\Dex(k)\setminus\Dex (k-1)\subseteq \Ball_2(v^k) \cup \{v^k \}$, so that
	\begin{equation}
	D^+(k)-D^+(k-1)\preceq Z_2+1.
	\end{equation}
Therefore, 
	\begin{equation}
	\E[D^+(k)]\leq k\E[Z_2+1]=2k(1+o(1)),
	\end{equation}
so that
	\begin{equation}\label{eq:expupb}
	D^+(tn^{2/3})=\Tp (n^{2/3}),
	\end{equation}
by the first moment method.
We define 
	\begin{equation}
	Z_2^{(\varepsilon)}\sim \Bin \Big((n-n^{2/3 +\varepsilon})\Bin \big(m-m^{2/3+\varepsilon},p)\big),p\Big).
	\end{equation}
From \eqref{eq:doulbebin} we see that
	\begin{equation}
	\big( X_n(k)\mid \Fil_n(k-1)\big) \succeq Z_2^{(\varepsilon)}, \quad \forall k\leq tn^{2/3} ,
	\end{equation}
when
	\begin{eqnarray}\label{eq:doubleup}
	Q(tn^{2/3}-1) \leq m^{2/3+\varepsilon},\quad  D^+(tn^{2/3}-1) \leq n^{2/3+\varepsilon} .
	\end{eqnarray}
By \eqref{eq:expupb} and \eqref{eq:upopp}, \eqref{eq:doubleup} is satisfied w.h.p. for every $t,\varepsilon >0$. Therefore, there exists a coupling between $(X_n(j))_{j\geq 1}$ and a sequence $(Z_2^{(\varepsilon)}(j))_{j \geq 1}$ of i.i.d. random variables with distribution identical to $Z_2^{(\varepsilon)}$ such that
	\begin{equation}
	\lim_{n \to \infty} \P(\exists j<tn^{2/3} : X_n(j)<Z_2^{(\varepsilon)}(j))=0.
	\end{equation}
	Consequently,
	\begin{equation}\label{eq:minzgeq}\begin{split}
	\limsup_{n \to \infty}&\ \P\Big(\min_{k\leq tn^{2/3}}S_n^\lambda(k)\leq -\varepsilon n^{2/3}/3\Big)\\ &\leq \limsup_{n \to \infty}\P\Big( \min_{k\leq tn^{2/3}} \sum_{j=1}^k (Z_2^{(\varepsilon)}(j)-1)\leq -\varepsilon n^{2/3}/3\Big),
	\end{split}\end{equation}
We next compute, for all $k\leq tn^{2/3}$,
	\begin{equation}
\E\Big[\sum_{j=1}^k( Z_2^{(\varepsilon)}(j)-1)\Big]=k p^2(m-m^{2/3+\varepsilon})(n-n^{2/3+\varepsilon})= \Theta (n^{1/3+ \varepsilon}),
\end{equation}
and
\begin{equation}
\Var\Big(\sum_{j=1}^k( Z_2^{(\varepsilon)}(j)-1)\Big)\leq k \E[(Z_2^{(\varepsilon)})^2]\leq k \E[Z_2^2]= k O(1)=O(n^{2/3}),
\end{equation}
so that, again, by Kolmogorov's inequality  (see e.g. \cite[Theorem 7.8.2]{GriSti92}), 
	\begin{equation}\begin{split}
	\P \Big(\min_{k \leq tn^{2/3}}&\sum_{j=1}^k (Z_2^{(\varepsilon)}(j)-1)< -\varepsilon n^{2/3}/3\Big)\\ &\leq \frac{\Var\Big(\sum_{j=1}^{tn^{2/3}} (Z_2^{(\varepsilon)}(j)-1)\Big)}{\Big(\min_{k \leq tn^{2/3}}\sum_{j=1}^k \E[(Z_2^{(\varepsilon)}(j)-1)]+\varepsilon n^{2/3}/3\Big)^2}\\
&=\frac{O(n^{2/3})}{\varepsilon^2 n^{4/3} (1-o(1))/9}\to 0.,
	\end{split}\end{equation}
 and, finally, by \eqref{eq:minzgeq}, for every $\varepsilon>0$,
	\begin{equation}\label{eq:actlow}
	\limsup_{n \to \infty}\ \P\Big(\min_{k\leq tn^{2/3}}S_n^\lambda(k)\leq -\varepsilon n^{2/3}/3\Big)=0.
	\end{equation}
Combining upper and lower bounds from \eqref{eq:actupp}  and \eqref{eq:actlow}, and using \eqref{eq:supactive}, the claim follows. \qed

\medskip

Now we can prove the two lemmas that establish the last two conditions for the application of the MFCLT, starting with the parabolic drift condition:

\begin{prop}[Parabolic drift condition]\label{lem:parabolic}
Consider the Doob's decomposition of the process $\big(S_n^\lambda (k)\big)_{k\geq 0}$ adapted to the two-step exploration process from Definition \ref{def:exploration}, starting from $v_0 \in \Ucal$, with $p= p_c (1+\lambda n^{-1/3})$ and $m=n^\alpha$, for some $\lambda \in \mathbb{R}, \alpha > 1$. Then as $n\to \infty$, for every $t \in (0,\infty)$,

	\begin{equation}
	\label{lem:cond(4)}
	n^{-1/3}\, \sup_{k \leq tn^{2/3}} \left| Y_n(k) - \frac{2\lambda k}{ n^{1/3} }
	+ \dfrac{k^2}{2n} \right| \pto 0.
	\end{equation}
\end{prop}
\proof
Recall from the Doob's decomposition of $(S_n^\lambda(k))_{k \geq 0}$ given in \eqref{eq:doob} that

	\begin{equation}
	Y_n(k)=\sum_{j=1}^k \E[X_n(j)\mid \Fil_n(j-1)].
	\end{equation} 
Consequently, to obtain the claim, we need to prove uniform convergence of the drift in the process at every time, i.e. that

	\begin{equation}\label{lem:drift}
	\max_{k\leq tn^{2/3}}\big|\E[X_n(k)\mid \Fil_n(k-1)]-(2\lambda -s)n^{-1/3} \big|= \op(n^{-1/3}).
	\end{equation}

First, we recall that $(|\Ball_1(v_k)\setminus \Opp(k-1)|\mid \Fil_n(k-1))$ has distribution $\Bin (m-Q(k-1),p)$, and that every element in $\Ball_1(v_k)\setminus \Opp(k-1)$ has $m-D^+(k-1)$ potential neighbours in $\Wcal \setminus \Dex (k-1)$, so that

	\begin{equation}
	\E[X_n(k)\mid \Fil_n(k-1)]=\E\Big[\Bin \Big(\big(n-D^+ (k-1)\big)\Bin \big(m-Q(k-1),p\big),p\Big)\Big]-1.
	\end{equation}
We compute
	\begin{align}\label{eq:decoexp}
	\E[X_n(k)\mid \Fil_n(k-1)]&=p^2(m-Q(k-1))(n-D^+(k-1))-1\\
	&=p^2mn-1 - p^2 D^+(k-1)m+p^2Q(k-1)(n-D^+(k-1)).\nonumber
	\end{align}
We know that, independently of $k$,
	\begin{equation}\label{eq:drifltlambda}
	p^2mn-1 = \dfrac{(1+\lambda n^{-1/3})^2}{mn}mn-1=2\lambda n^{-1/3}+ \lambda^2 n^{-2/3}.
	\end{equation}
Next, from Lemma \ref{lem:concopp}, we obtain

	\begin{equation}\begin{split}\label{eq:driftopp}
	\max_{k\leq tn^{2/3}}p^2Q(k-1)(n-D^+(k-1))&\leq  \max_{k\leq tn^{2/3}} p^2n Q(k-1)\\
	&= p^2n tm^{1/2}n^{1/6}(1+\op(1))\\
	&= \Tp(n^{1/6}m^{-1/2})=\op(n^{-1/3}).
	\end{split}\end{equation}
What remains to prove, since $sn^{-1/3}=k/n$, is that 

	\begin{equation}\label{eq:driftdead}
	\max_{k\leq t n^{2/3}}\big|p^2 D^+(k-1)m-k/n\big|=\op(n^{-1/3}).
	\end{equation}
For the lower bound on $p^2 D^+(k-1)m$ we note that

	\begin{equation}
	D^+(k-1)=D(k-1)+A(k-1)\geq D(k-1)=k-1,
	\end{equation}
so that, almost surely, for all $k\leq tn^{2/3}$,
	\begin{equation}\begin{split}
	p^2 D^+(k-1)m-k/n & \geq (k-1) (p^2m- n^{-1})-n^{-1}\\&=(k-1)\Big(\frac{1+2\lambda n^{-1/3}+n^{-2/3}}{mn}m-n^{-1}\Big)-n^{-1}\\&=(k-1)\dfrac{2\lambda+n^{-1/3}}{n^{4/3}}-n^{-1}=o(n^{-1/3}).
	\end{split}\end{equation}
For the matching upper bound, using Lemma \ref{lem:supactive} and the fact that $D(k-1)=k-1$, we obtain that
	\begin{equation}
	\max_{k\leq tn^{2/3}}\dfrac{D^+(k-1)-k}{n^{2/3}}= \max_{k\leq tn^{2/3}} \dfrac{A(k-1)-1}{n^{2/3}}\pto 0,
	\end{equation}
so that

\begin{equation}\begin{split}
\max_{k \leq tn^{2/3}}\Big(p^2 D^+(k-1)m-n^{-1/3}s\Big)&\leq  (k-1) (p^2m- n^{-1})-n^{-1} +p^2m\max_{k\leq tn^{2/3}} A(k)\\&=\op(n^{-1/3})+\frac{1+o(1)}{n}\max_{k\leq tn^{2/3}} A(k)=\op(n^{-1/3}),
\end{split}\end{equation}
and \eqref{eq:driftdead} follows.
Substituting \eqref{eq:drifltlambda}, \eqref{eq:driftopp} and \eqref{eq:driftdead} into \eqref{eq:decoexp} we obtain  \eqref{lem:drift}.

From \eqref{lem:drift} we obtain, summing over $j$, by the uniform convergence,
	\begin{equation}
	n^{-1/3}\, \sup_{k \leq tn^{2/3}} \left| Y_n(k) - \frac{2\lambda k }{n^{1/3} }
	+ \dfrac{k^2}{2n} \right|\leq  n^{-1/3}tn^{2/3}\op(n^{-1/3}) \pto 0.
	\end{equation}

\qed
%For an upper bound we use second moment method. We know that
%
%\begin{equation}
%D^+(k)\preceq \sum_{j=1}^k Z_2(j),
%\end{equation}
%where $Z_2(j)$ are i.i.d. random variables with the same distribution as $Z_2$.
%From this we obtain
%
%\begin{equation}
%\E\Big[\sum_{j=1}^k  Z_1\Big]=k\E[Z_1]=kpm=k\sqrt{\dfrac{m}{n}}(1+o(1)),
%\end{equation}
\subsection{Linear variance condition}
\label{sec:variance}

The only condition left to prove is the linear variance condition. We do it in the following lemma:
\begin{lem}[Linear variance condition]\label{lem:vari} 
Consider the Doob's decomposition of the process $\big(S_n^\lambda (k)\big)_{k\geq 0}$ adapted to the two-step exploration process from Definition \ref{def:exploration}, starting from $v_0 \in \Ucal$, with $p= p_c (1+\lambda n^{-1/3})$ and $m=n^\alpha$ for some $\lambda \in \mathbb{R}, \alpha > 1$. Then as $n\to \infty$, for every $t \in (0,\infty)$,
	\begin{equation}
	n^{-2/3}\,L_n(tn^{2/3}) \pto t.
	\end{equation}
\end{lem}

\proof

From the Doob's decomposition of $(S_n^\lambda(k))_{k \geq 0}$ given in \eqref{eq:doob}  we recall that
	\begin{equation}
	L_n(k)=\sum_{j=1}^k \Var (X_n(j)\mid \Fil_n(j-1)).
	\end{equation}
We recall from Definition \ref{def:exploration} that we can write $X_n(j)+1$ as a binomial random variable with random parameter. Indeed,
	\begin{equation}\begin{split}
	X_n(j):&= | \Ball_2(v_j,\Opp(j-1)\cup \Dex (j-1))|-1
	\\&\overset{d}{=} \Bin \Big(\big(Q(j)-Q(j-1)\big)\big(n-D^+(j-1)\big),p\Big)-1.
	\end{split}\end{equation}
We can thus use the variance decomposition formula on the conditioning on $Q(j)$, to obtain
	\begin{equation}\begin{split}\label{eq:variancedecomposition}
	\Var (X_n(j)\mid \Fil_n(j-1))=& \Var(\E[X_n(j)\mid Q(j),\Fil_n(j-1)]\mid \Fil_n(j-1)])
	\\&+\E[\Var(X_n(j)\mid Q(j),\Fil_n(j-1))\mid \Fil_n(j-1))].
	\end{split}\end{equation}
We can thus bound

	\begin{equation}\begin{split}\label{eq:maxdeco}
	\max_{j\leq tn^{2/3}}&\big|\Var (X_n(j)\mid \Fil_n(j-1))-1\big|
	\\ \leq &\max_{j\leq tn^{2/3}}\Var(\E[X_n(j)\mid Q(j),\Fil_n(j-1)]\mid \Fil_n(j-1)])
	\\ &+\max_{j\leq tn^{2/3}}\big| \E[\Var(X_n(j)\mid Q(j),\Fil_n(j-1))\mid \Fil_n(j-1))]-1 \big|.
	\end{split}\end{equation}
We next compute that
	\begin{equation}
	\E[X_n(j)\mid Q(j), \Fil_n(j-1)]=p(n-D^+(j-1)) (Q(j)-Q(j-1))-1,
	\end{equation}
so that, since $Q(j-1)$ and $D^+(j-1)$ are both measurable with respect to $\Fil_n(j-1)$, and recalling that $(Q(j)-Q(j-1)\mid \Fil_n(j-1)\deq \Bin(m-Q(j-1),p)$, we obtain
	\begin{equation}\begin{split}\label{eq:varE}
	\max_{j\leq tn^{2/3}}&\Var(\E[X_n(j)\mid Q(j),\Fil_n(j-1)]\mid \Fil_n(j-1)])\\
	&=p^2\max_{j\leq tn^{2/3}}\Big\{\big(n-D^+(j-1)\big)^2\Var\big(Q(j)-Q(j-1)\mid \Fil_n(j-1)\big)\Big\}\\
	&\leq p^2 \max_{j\leq tn^{2/3}}\big(n-D^+(j-1)\big)^2 \max_{j\leq tn^{2/3}}p(1-p)(m-Q(j-1))\\
	&\leq p^3 n^2m  \to 0.
	\end{split}\end{equation}
For the second term on the right hand side of \eqref{eq:maxdeco} we compute
	\begin{equation}\label{eq:varup}\begin{split}
	\Var(X_n(j)&\mid Q(j),\Fil_n(j-1))\\
	&= \Var \big(\Bin \big((n-D^+(j-1))(Q(j)-Q(j-1)),p\big)\big)\\
	&=p(1-p)\big((n-D^+(j-1))(Q(j)-Q(j-1))\big),
	\end{split}\end{equation}
so we write, recalling again that $D^+(j-1)$ is measurable with respect to $\Fil_n(j-1)$,

	\begin{equation}\begin{split}
	\E[\Var(X_n(j)\mid & Q(j),\Fil_n(j-1))\mid \Fil_n(j-1))]\\&=p(1-p)(n-D^+(j-1))\E[Q(j)-Q(j-1)\mid\Fil_n(j-1)]\\&=p^2(1-o(1))\big(n-D^+(j-1)\big)\big(m-Q(j-1)\big).
	\end{split}\end{equation}
	As a result, for an upper bound,
	\begin{equation}
	\max_{j\leq tn^{2/3}}\E[\Var(X_n(j)\mid  Q(j),\Fil_n(j-1))\mid \Fil_n(j-1))] - 1\leq p^2mn(1-o(1))-1\to 0.
	\end{equation}
To derive the matching lower bound, we recall from Lemma \ref{lem:concopp} that $\max_{j\leq tn^{2/3}}Q(j)= \Tp (m^{1/2}n^{1/6})$ and from Lemma \ref{lem:supactive} that $\max_{j\leq tn^{2/3}}D^+(j)= \Tp (n^{2/3})$, so that

	\begin{equation}\begin{split}\label{eq:varlow}
	\min_{j\leq tn^{2/3}}&\E[\Var(X_n(j)\mid  Q(j),\Fil_n(j-1))\mid \Fil_n(j-1))] - 1 \\ &\geq p^2(1-o(1))(n- \Tp (n^{2/3}))(m-\Tp (m^{1/2}n^{1/6}) ) -1\pto 0.
	\end{split}\end{equation}
	We thus obtain, combining upper and lower bounds from \eqref{eq:varup} and \eqref{eq:varlow},
	\begin{equation}\label{eq:Evar}
	\max_{j\leq tn^{2/3}}\big|\E[\Var(X_n(j)\mid  Q(j),\Fil_n(j-1))\mid \Fil_n(j-1))] - 1\big|\pto 0.
	\end{equation}
Substituting the bounds from \eqref{eq:varE} and \eqref{eq:Evar} into \eqref{eq:maxdeco} we obtain
	\begin{equation}
	\max_{j\leq tn^{2/3}}\big|\Var(X_n(j)\mid \Fil_n(j-1)) - 1\big| \pto 0.
	\end{equation}
Summing over $j\leq tn^{2/3}$ the claim follows, by uniformity of the convergence. \qed
\medskip

We can now finally complete the proof of Theorem \ref{the:brownian}:

\begin{proof}[Proof of Theorem \ref{the:brownian}]
We apply Theorem \ref{thm:MFCLTcond} with $(S_n(k))_{k \geq 0}=(S_n^\lambda(k))_{k \geq 0}$ and $\tau =\lambda$.
Conditions (1) and (2) are satisfied by Lemma \ref{lem:12}, Condition (3) by Proposition \ref{lem:parabolic} and Condition (4) by Lemma \ref{lem:vari}. We can then conclude that

\begin{equation}
	\overline{S}^\lambda_n (s)=\big(n^{-1/3}S_n^\lambda( sn^{2/3})\big)_{s \geq 0}  \dto 
	\left(W(s) + 2\lambda s - \frac{s^2}{2}\right)_{s \geq 0}.
	\end{equation}
\end{proof}

\section{Scaling limit of number of edges in large critical components}
\label{sec:edge}
In this section we prove that in the largest connected components of the critical RIG $G(n,m,p)$, when $\alpha < 1$, the total number of edges is of higher order of magnitude than the number of vertices,  i.e., for every $j$,
	\begin{equation}
	|\cluster_j|=\op(|\Ed(\cluster_j)|).
	\end{equation}
 The reason is that if we assume that $\alpha <1$, $G(n,m,p)$ is built by planting cliques whose sizes are close to $\sqrt{n/m}$, each of which contains around $n/(2m)$ edges.

\bigskip

We now prove Proposition \ref{prop:surpll}.
%We show that in a large critical connected component actually the number of edges is of higher order of magnitude than the number of vertices, i.e., for all $j \in \mathbb{N}$
Because the edges in the RIG are generated on the base of how individuals are assigned to communities, we first prove that w.h.p. $G(n,m,p)$ does not contain unusually large communities. We recall that for each community $w \in \Wcal$, the number of its elements is given by
	\begin{equation}\label{eq:ballindi}
	\partial B_1(w)=\sum_{v\in \Ucal}\indi_{\{\{v,w\} \text{ is open}\}}\deq \Bin (n,p).
	\end{equation}

To prove that there are no exceptionally large communities we use results from standard concentration inequalities for binomial random variables:
\begin{thm}\cite[Theorem 2.21]{Hofs17}
Let $X\deq \Bin (n,p)$. Then for every $t>0$

	\begin{align}
%	\P(\E[S]-S \geq t)&\leq \exp\Big\{\dfrac{-t^2}{2\sum_{i=1}^m \E[X_i^2]}\Big\},	\label{eq:lowdev}\\
	\P(X-np \geq t)&\leq \exp\Big\{\dfrac{-t^2}{2np+2td/3}\Big\}\label{eq:updev}.
	\end{align}
\end{thm}

We next prove Lemma \ref{lem:commsize}, i.e., that the graph does not contain communities that are significantly larger than the average:

\begin{proof}[Proof of Lemma \ref{lem:commsize}]
We only prove \eqref{eq:commsize1}, the proof of \eqref{eq:commsize2} is identical with the roles of $m$ and $n$ switched.
We know that $\partial B_1(w)\deq \Bin (n,p)$.
%We can thus prove the lower bound using \eqref{eq:lowdev} to compute
%	\begin{equation}\begin{split}
%	\P(\E[ B_1(w)]-  B_1(w)\geq \varepsilon np) &\leq \exp \Big\{- \dfrac{(\varepsilon np)^2}{2 \sum_{v\in \Ucal} \E[\indi_{\{(v,w) \text{ is open}\}}^2]}\Big\}
%	\\&= \exp\Big\{- \dfrac{\varepsilon^2n^2p^2}{2np}\Big\}=\exp\Big\{-\dfrac{\varepsilon^2np}{2}\Big\}.
%	\end{split}\end{equation}
%We know that
%
%	\begin{equation}\begin{split}
%	\P\Big(\min_{w\in\Wcal} \Big(\dfrac{\partial B_1(w)}{np}-1\Big)\leq -\varepsilon\Big)&\leq n\P(\E[ B_1(w)]- B_1(w) \geq \varepsilon np) \geq \varepsilon np)
%	\\ &\leq m\exp\Big\{-\dfrac{\varepsilon^2np}{2}\Big\} \to 0.
%	\end{split}\end{equation}
We use \eqref{eq:updev} with $t=\varepsilon n p$ for some $\varepsilon >0$ small enough,

	\begin{equation}\begin{split}
	\P(  \partial B_1(w)-np\geq \varepsilon np) &\leq \exp\Big\{-\dfrac{-(\varepsilon np)^2}{2np+2\varepsilon np/3}\Big\}
	\leq \exp\Big\{-\dfrac{3\varepsilon^2np}{8}\Big\}.
	\end{split}\end{equation}
Thus, we write
	\begin{equation}\begin{split}
	\P\Big(\max_{w\in\Wcal} \Big\{\dfrac{\partial B_1(w)}{np}-1\Big\}\geq \varepsilon\Big)&\leq m\P(\partial B_1(w)-np \geq \varepsilon np)\\&\leq m\exp\Big\{-\dfrac{3\varepsilon^2np}{8}\Big\} \to 0.
	\end{split}\end{equation}
\end{proof}

This result is important since when we run the two-steps exploration on $K_p(n,m)$ starting from $v_0 \in \Wcal$, as we do to study $G(n,m,p)$ when $\alpha<1$, the process explores one community at every step, so from the sizes of the communities explored in the first $k$ steps we can bound the number of edges explored.

%\begin{prop}\label{prop:surpll}
%Consider the two-steps exploration process on $K_p(m,\overline{m})$ with $\overline{m}=m^\alpha$, $\alpha >1$, and $p=p_c(1+\lambda m^{-1/3})$. Define the process $\big(E (k)\big)_{k \geq 1}$ as the process that counts the number of edges in the corresponding RIG $G(n,m,p)=G(n,m,p)$ found by the exploration. Then for every  $T>0$
%
%	\begin{equation}
%	\sup_{t\leq T}\Big| \dfrac{E(tn^{2/3})}{tnm^{-1/3}}-\dfrac{1}{2}\Big| \pto 0.
%	\end{equation}
%
%\end{prop}
\begin{proof}[Proof of Proposition \ref{prop:surpll}]

We start noting that for every community $v \in \Wcal$, $ \Ball_1(v)$ induces a clique in the RIG, and that every edge of $G(n,m,p)$ must belong to at least one such clique. From this we can deduce that

	\begin{equation}
	E(k)\leq \sum_{j\leq k} \dfrac{\partial B_1(v_j)(\partial B_1(v_j)-1)}{2}.
	\end{equation}
Indeed, in this upper bound, we are ignoring the fact that the exploration might have already found some of the edges in the $j$-th community in the previous steps.
By Lemma \ref{lem:commsize},
	\begin{equation}
	\sum_{j\leq k}\dfrac{\partial B_1(v_j)(\partial B_1(v_j)-1)}{2} \leq k \dfrac{np(np-1)}{2}(1+\op(1))= \dfrac{kn}{2m}(1+\op(1)).
	\end{equation}
We thus obtain that

	\begin{equation}\label{eq:Eup}
	E(tm^{2/3})\leq \dfrac{tn}{2m^{1/3}}(1+\op(1)).
	\end{equation}

%We know that
%\begin{equation}
%\big(\partial B_1(v_j)\mid \Fil_n(j-1)\big) \preceq 1+ \Bin (n,p),
%\end{equation}
%since one individual of the community has to be found to explore it, and all the other individuals are part of the community corresponding to to $v_j$ independently with probability $p$.
%
For a lower bound instead, we note that all the edges added in the community corresponding to $v^k$ among vertices in $\Ucal\setminus \Opp(k-1)$, are in $\Ed(k)\setminus \Ed(k-1)$, since such vertices have not been found yet by the exploration process, so that none of the edges incident to them has been explored. Thus we write, recalling the definition of $\partial B_1(v_j, \Opp(j-1))$ from \eqref{eq:ballout},

	\begin{equation}\label{eq:lowedge}
	E(j)-E(j-1)\geq \dfrac{\partial B_1(v_j, \Opp(j-1))(\partial B_1(v_j,\Opp (j-1))-1)}{2},
	\end{equation} 
and we note that 
\begin{equation}
(\partial B_1(v_j, \Opp(j-1))\mid \Fil_n(j-1))\deq \Bin (n-\Opp(j-1),p).
\end{equation}
We define the random variable  $\underline{B}$ as
	\begin{equation}
	\underline{B}\deq\Bin(n-2tn^{1/2}m^{1/6},p).
	\end{equation}
Note that under the assumption that 
\begin{equation}\label{eq:oppdom}
Q(tm^{2/3}-1)\leq2tn^{1/2}m^{1/6},
\end{equation}
it holds, for every given $t \in (0,\infty)$,

\begin{equation}
(\partial B_1(v_j, \Opp(j-1))\mid \Fil_n(j-1))\succeq \underline{B}, \quad \forall j \leq tm^{2/3}.
\end{equation}
We then obtain that, by \eqref{eq:lowedge}, if \eqref{eq:oppdom} is satisfied,  which happens  w.h.p. by  Lemma \ref{lem:concopp}, then

\begin{equation}\label{eq:dominationwhp}
(E(j)-E(j-1)\mid \Fil_n(j-1))\succeq \dfrac{\underline{B}(\underline{B}-1)}{2}, \quad \forall j\leq tm^{2/3}.
\end{equation}
Thus, there exists a coupling between $(E(j))_{j \geq 1}$ and a sequence of i.i.d. random variables  $(\underline{B}(j))_{j \geq 1}$, each with distribution identical to $\underline{B}$, such that
	\begin{equation}\label{eq:dominationwhp}\begin{split}
	\P\Big(\exists j \leq  tm^{2/3}: E(j)-E(j-1)<&\dfrac{\underline{B}(j)(\underline{B}(j)-1)}{2}\Big) \\ &\leq \P( Q(tm^{2/3}-1)>2tn^{1/2}m^{1/6}) \to 0.
	\end{split}\end{equation}
%	\begin{equation}
%	E(k) \preceq \sum_{j\leq k} \dfrac{\underline{B}(j)(\underline{B}(j)-1)}{2} 
%	\end{equation}
%where $\big(\underline{B}(j)\big)_{j\geq 1}$ is a sequence of i.i.d. random variables with
%	\begin{equation}
%	\underline{B}(j):=\Bin(n-n^{2/3},p).
%	\end{equation}
We recall the formulas for the third and fourth moments of binomial random variables:
\begin{align}
\E[\Bin(n,p)^3]&=np(1-3p+3np+2p^2-3np^2+n^2p^2)	\\
\E[\Bin(n,p)^4]&=np(1-7p+7np+12p^2-18np^2+6n^2p^2-6p^3+11np^3-6n^2p^3+n^3p^3),	\nonumber
\end{align}
so that we can bound 
	\begin{align}
	\E[\underline{B}(j)(\underline{B}(j)-1)]&=(np)^2(1+o(1))=n/m(1+o(1)),\\
	\Var(\underline{B}(j)(\underline{B}(j)-1))&=6(np)^2(1+o(1))=6n/m(1+o(1)),
	\end{align}
so that, by the second moment method, for every $t \in (0,\infty)$, 
	\begin{equation}
	\sum_{j\leq tm^{2/3}} \dfrac{\underline{B}(j)(\underline{B}(j)-1)}{2}=\dfrac{tn}{2m^{1/3}}(1+\op(1)).
	\end{equation}
From \eqref{eq:dominationwhp} we know that, for every $\varepsilon,t >0$,

\begin{equation}\label{eq:Elow}\begin{split}
\limsup_{n \to \infty}\P\Big(E(tm^{2/3})&\leq \dfrac{(t-\varepsilon)n}{2m^{1/3}}\Big) \\ &\leq \limsup_{n \to \infty} \P\Big(\sum_{j\leq tm^{2/3}} \dfrac{\underline{B}(j)(\underline{B}(j)-1)}{2} \leq \dfrac{(t-\varepsilon)n}{2m^{1/3}}\Big)= 0.
\end{split}\end{equation}

Combining \eqref{eq:Eup} and \eqref{eq:Elow} we conclude that, for every $t\in (0,\infty)$,

	\begin{equation}
	\dfrac{E(tm^{2/3})m^{1/3}}{n }-\dfrac{t}{2}\pto 0.
	\end{equation}
We obtain the claimed uniform convergence from the fact that $(E(k))_{k \geq 0}$ is a non-decreasing process and the function $t \mapsto t/2$ is continuous.
\end{proof}

%\section{Proof of convergence in the $\ell^2$ topology}\label{sect:l2}
%
%In this section we prove that the convergences in distribution in \eqref{eq:alfageq}, \eqref{eq:alfamin} and \eqref{eq:bipa} hold in stronger topologies than just the product topology over $\mathbb R_+^\infty$. 
%
%\begin{proof}[Proof of  \eqref{eq:alfageq} in the $\ell^2_\searrow$-topology]
%
%
%\end{proof}
%
%
%\begin{proof}[Proof of  \eqref{eq:alfamin} in the $\ell^2_\searrow\times \ell^2$-topology]
%
%\end{proof}
%
%\begin{proof}[Proof of  \eqref{eq:bipa} in the $\ell^2_\searrow$-topology]
%
%\end{proof}

	\section*{Acknowledgments}

The work in this paper is supported by the Netherlands Organisation for Scientific Research 
(NWO) through Gravitation-grant NETWORKS-024.002.003 and by the European Research Council (ERC) through Starting Grant Random Graph, Geometry and Convergence 639046. 
%%%%%%%%%%% REFERENCES %%%%%%%%%%%%

\begin{small}
\bibliographystyle{abbrv}
\bibliography{LorenzosBib}
\end{small}
%%%%%%%%%%%%%%%%%%%%%%%%%%%%%%%

\bigskip
\end{document}